\documentclass{article}
\usepackage{amsmath,amssymb,amsthm,color,graphicx,diagbox,tikz,colortbl,array}
\usepackage{setspace}
\usepackage{float}

\title{Modal logics of conjunctively closed provability predicates}
\author{Haruka Kogure\footnote{Email: kogure@stu.kobe-u.ac.jp}
\footnote{Graduate School of System Informatics, Kobe University, 1-1 Rokkodai, Nada, Kobe 657-8501, Japan.}
and Taishi Kurahashi\footnote{Email: kurahashi@people.kobe-u.ac.jp}
\footnote{Graduate School of System Informatics, Kobe University, 1-1 Rokkodai, Nada, Kobe 657-8501, Japan.}}
\date{}

\theoremstyle{plain}
\newtheorem{thm}{Theorem}[section]
\newtheorem*{thm*}{Theorem}
\newtheorem{lem}[thm]{Lemma}
\newtheorem{prop}[thm]{Proposition}
\newtheorem{cor}[thm]{Corollary}
\newtheorem{fact}[thm]{Fact}
\newtheorem*{fact*}{Fact}
\newtheorem{prob}[thm]{Problem}
\newtheorem*{prob*}{Problem}
\newtheorem*{cl*}{Claim}
\newtheorem{cl}{Claim}[section]
\newtheorem*{scl*}{Subclaim}

\theoremstyle{definition}
\newtheorem{defn}[thm]{Definition}

\newcommand{\PL}{\mathsf{PL}}

\newcommand{\Sub}{\mathsf{Sub}}
\newcommand{\MF}{\mathsf{MF}}
\newcommand{\tc}{\vdash^{\mathrm{tc}}}

\newcommand{\clo}{\mathrm{cl}}

\newcommand{\PA}{\mathsf{PA}}
\newcommand{\PR}{\mathrm{Pr}}

\newcommand{\Prov}{\mathrm{Prov}}
\newcommand{\Proof}{\mathrm{Proof}}
\newcommand{\Con}{\mathrm{Con}}

\newcommand{\gn}[1]{\ulcorner#1\urcorner}

\newcommand{\D}[1]{\mathbf{D#1}}

\newcommand{\num}{\overline}

\newcommand{\LA}{\mathcal{L}_A}

\newcommand{\ConS}{\mathrm{Con}^{\mathrm{S}}}

\newcommand{\ConL}{\mathrm{Con}^{\mathrm{L}}}

\newcommand{\dg}{\mathrm{d}}

\newcommand{\Th}{\mathrm{Th}}

\newcommand{\N}{\mathsf{N}}
\newcommand{\CN}{\mathsf{CN}}
\newcommand{\CNF}{\mathsf{CN4}}
\newcommand{\CNP}{\mathsf{CNP}}
\newcommand{\CND}{\mathsf{CND}}
\newcommand{\CNPF}{\mathsf{CNP4}}
\newcommand{\CNDF}{\mathsf{CND4}}
\newcommand{\GL}{\mathsf{GL}}

\begin{document}

\maketitle

\begin{abstract}
We investigate non-normal modal logics corresponding to provability predicates $\PR_T(x)$ satisfying the derivability condition $\mathbf{C}$: $T\vdash\PR_T(\gn{\varphi})\land\PR_T(\gn{\psi})\to \PR_T(\gn{\varphi\land\psi})$. 
The modal counterpart of this condition is the axiom scheme $\mathsf{C}$: $\Box A\land\Box B\to\Box(A\land B)$. 
First, we introduce a new semantics based on closure operators for non-normal modal logics including logics adopting $\mathsf{C}$ as an axiom scheme. 
We prove modal completeness for several non-normal modal logics studied in this paper with respect to this semantics.
Second, we prove the arithmetical completeness theorems for the logics $\CN$, $\CNP$, $\CNF$, $\CNPF$, and $\CND$ by using our new semantics.
\end{abstract}

\section{Introduction}

Let $T$ be a primitive recursively axiomatized consistent $\LA$-theory extending Peano Arithmetic $\PA$, where $\LA$ is the language of first-order arithmetic. 
In the usual proof of G\"odel's incompleteness theorems, provability predicates of a theory $T$ play important roles. 
The theorem states that if a provability predicate $\PR_T(x)$ of $T$ satisfies the following Hilbert--Bernays--L\"ob derivability conditions $\D{2}$ and $\D{3}$, then $T$ cannot prove the consistency statement $\Con_T : \equiv \neg \PR_T(\gn{0=1})$ of $T$. 
\begin{itemize}
    \item $\D{2}$: $T \vdash \PR_T(\gn{\varphi \to \psi}) \to (\PR_T(\gn{\varphi}) \to \PR_T(\gn{\psi}))$. 
    \item $\D{3}$: $T \vdash \PR_T(\gn{\varphi}) \to \PR_T(\gn{\PR_T(\gn{\varphi})})$. 
\end{itemize}
These conditions naturally correspond to the following modal principles $\mathsf{K}$ and $\mathsf{4}$ respectively by interpreting the modal operator $\Box$ as $\PR_T(x)$. 
\begin{itemize}
    \item $\mathsf{K}$: $\Box(A \to B) \to (\Box A \to \Box B)$. 
    \item $\mathsf{4}$: $\Box A \to \Box \Box A$. 
\end{itemize}
For each provability predicate $\PR_T(x)$ of $T$, let $\PL(\PR_T)$ be the set of all $T$-verifiable modal principles of $\PR_T(x)$. 
Solovay's arithmetical completeness theorem \cite{Solovay} states that if $T$ is $\Sigma_1$-sound, then for the canonical $\Sigma_1$ provability predicate $\Prov_T(x)$ of $T$, $\PL(\Prov_T)$ is exactly the G\"odel--L\"ob modal logic $\GL$.

On the other hand, not all provability predicates satisfy the full
Hilbert--Bernays--L\"ob derivability conditions. 
For instance, it is known that the second incompleteness theorem does not hold for Rosser provability predicates. 
Hence every Rosser provability predicate does not satisfy at least one of $\D{2}$ and $\D{3}$. 
Later, Mostowski \cite{Mos65} introduced a particularly simple example of a $\Sigma_1$ provability predicate for which the second incompleteness theorem fails. Let $\PR_T^{\mathrm{M}}(x) :\equiv \PR_T(x)\land x\neq \gn{0=1}$.
Then $T$ proves the corresponding consistency statement $\neg\PR_T^{\mathrm{M}}(\gn{0=1})$. 
On the other hand, $\PR_T^{\mathrm{M}}(x)$ still satisfies the derivability condition $\D{3}$, and hence it cannot satisfy $\D{2}$. 
Then, the corresponding provability logic $\PL(\PR_T^{\mathrm{M}})$ is a non-normal modal logic.

The pure logic of necessitation $\N$, which was introduced by Fitting, Marek, and Truszczy\'nski \cite{FMT}, is obtained from classical propositional logic by adding only the necessitation rule $\dfrac{A}{\Box A}$ in the language of modal propositional logic. 
The rule corresponds to the derivability condition $\mathbf{D1}: T \vdash \varphi \Rightarrow T \vdash \PR_T(\gn{\varphi})$, which is common to all provability predicates. 
In fact, the second author proved in \cite{Kura24} that $\N$ is exactly the provability logic of all provability predicates. 
Moreover, the arithmetical completeness of several extensions of $\N$ obtained by adding modal principles corresponding to derivability conditions and consistency principles has been established in our previous works \cite{Kogure1,Kogure2,KK,KK2,Kura24}. 
An overview of this line of research is given in \cite{KK2}.

The purpose of the present paper is to continue the research by studying provability predicates satisfying the derivability condition $\mathbf{C}$: $T\vdash\PR_T(\gn{\varphi})\land\PR_T(\gn{\psi})\to \PR_T(\gn{\varphi\land\psi})$. 
Its modal counterpart is the axiom scheme $\mathsf{C}$: $\Box A\land\Box B\to\Box(A\land B)$. 
The significance of the condition $\mathbf{C}$ in the context of the second incompleteness theorem was investigated in \cite{Kura25}. 
In particular, the combination of $\mathbf{C}$ with $\mathbf{D3}$ yields a version of the second incompleteness theorem on the unprovability of the schematic consistency statement $\ConS_T : = \{\neg (\PR_T(\gn{\varphi}) \land \PR_T(\gn{\neg \varphi})) \mid \varphi$ is an $\LA$-sentence$\}$. 
Mostowski's predicate $\PR_T^{\mathrm{M}}(x)$ is an example satisfying both $\mathbf{C}$ and $\mathbf{D3}$. 
Hence, it satisfies $T\vdash \neg \PR_T^{\mathrm{M}}(\gn{0=1})$, but $T \nvdash \neg(\PR_T^{\mathrm{M}}(\gn{\varphi})\land \PR_T^{\mathrm{M}}(\gn{\neg\varphi}))$
for some $\LA$-sentence $\varphi$. 
Thus, while the modal principles $\mathsf{P}: \neg\Box\bot$ and $\mathsf{D}: \neg(\Box A\land\Box\neg A)$ are equivalent in normal modal logics, they are not equivalent in general in the non-normal setting. 
Indeed, by the above mentioned version of the second incompleteness theorem, there is no provability predicate $\PR_T(x)$ such that $\PL(\PR_T)$ contains $\CNDF$, whereas, $\PL(\PR_T^{\mathrm{M}})$ contains $\CNPF$. 

One obstacle to studying the principle $\mathsf{C}$ over the logic $\N$ is that the relational semantics introduced by Fitting, Marek, and Truszczy\'nski does not seem to handle $\mathsf{C}$ satisfactorily. 
Therefore, in the present paper, we introduce a simple new semantics suitable for $\mathsf{C}$. 
Our semantics consists of points, but it does not employ an accessibility relation. 
Instead, each point $x$ is assigned a set $G(x)$ of modal formulas, and a fixed closure operator $\clo$ determines the set $\clo(G(x))$ of modal formulas. 
The formula $\Box A$ is true at $x$ exactly when $A\in\clo(G(x))$.
This separation between the generator $G(x)$ and its closure $\clo(G(x))$ is useful for finite countermodel constructions. 
Even when the set of points and all generators $G(x)$ are finite, the set $\clo(G(x))$ may be infinite. 
The modal principles are then reflected by conditions on the closure $\clo(G(x))$. 
For example, $\mathsf{C}$ is true at $x$ if and only if $\clo(G(x))$ is closed under taking conjunction. 

In Section~\ref{Sec3}, we prove completeness theorems with respect to our new semantics for non-normal modal logics having combinations of the modal principles $\mathsf{N}$, $\mathsf{C}$, $\mathsf{P}$, $\mathsf{D}$, and $\mathsf{4}$. 
For logics in which $\mathsf{N}$ and $\mathsf{D}$ are not both contained, we actually prove the finite model property. 
For logics containing both $\mathsf{N}$ and $\mathsf{D}$, we prove completeness by a canonical model construction. 
However, the countermodels obtained in this way are not primitive recursively presented.
For arithmetical completeness, in order to embed countermodels into arithmetic, we need a primitive recursive presentation of the countermodels. 
We therefore prove separately that the set of all theorems of $\CND$ is primitive recursive, and we then prove that if $\CND\nvdash A$, then we can find a primitive recursively presented countermodel for $A$. 

In Section~\ref{Sec4}, we then turn to arithmetical completeness. 
We prove the arithmetical completeness theorems for $\CN$, $\CNP$, $\CNF$, $\CNPF$, and $\CND$ based on our new semantics. 
The proof method is based on the Solovay-style constructions for non-normal provability logics developed in our previous studies (see \cite{Kogure1,Kogure2,KK,KK2,Kura24}).

\section{Preliminaries}\label{Sec2}

In this section, we introduce basic notions and notation used throughout the paper. 

\subsection{Provability predicates and derivability conditions}

Throughout the present paper, $T$ denotes a primitive recursively axiomatized consistent extension of Peano Arithmetic $\PA$ in the language $\LA$ of first-order arithmetic (cf.~H\'ajek and Pudl\'ak \cite{HP}). 
Let $\omega$ denote the set of all natural numbers. 
For each $n\in\omega$, the numeral for $n$ is denoted by $\num{n}$. 
We fix a standard G\"{o}del numbering such that if $\alpha$ is a proper sub-expression of a finite sequence $\beta$ of $\LA$-symbols, 
then the G\"{o}del number of $\alpha$ is less than that of $\beta$. 
For each $\LA$-formula $\varphi$, let $\gn{\varphi}$ denote the numeral for the G\"odel number of $\varphi$.

A $\Sigma_1$ formula $\PR_T(x)$ is called a provability predicate of $T$ if for every $\LA$-sentence $\varphi$, $T\vdash\varphi$ if and only if $\PA\vdash\PR_T(\gn{\varphi})$. 

\begin{defn}
Let $\PR_T(x)$ be a provability predicate of $T$.
\begin{enumerate}
\item $\PR_T(x)$ is said to satisfy $\D{2}$ if $T \vdash \PR_T(\gn{\varphi \to \psi})\to (\PR_T(\gn{\varphi}) \to \PR_T(\gn{\psi}))$ for any $\LA$-sentences $\varphi$ and $\psi$.

\item $\PR_T(x)$ is said to satisfy $\D{3}$ if $\PA\vdash\PR_T(\gn{\varphi})\to \PR_T(\gn{\PR_T(\gn{\varphi})})$ for any $\LA$-sentence $\varphi$.

\item $\PR_T(x)$ is said to satisfy $\mathbf{C}$ if $\PA\vdash\PR_T(\gn{\varphi})\land\PR_T(\gn{\psi})\to \PR_T(\gn{\varphi\land\psi})$ for any $\LA$-sentences $\varphi$ and $\psi$.
\end{enumerate}
\end{defn}

It is easily shown that if $\PR_T(x)$ satisfies $\D{2}$, then it also satisfies $\mathbf{C}$. 

\begin{defn}
Let $\PR_T(x)$ be a provability predicate of $T$.
\begin{enumerate}
\item $\ConL_T$ is the sentence $\neg\PR_T(\gn{0=1})$.

\item $\ConS_T$ is the set
$\{\neg (\PR_T(\gn{\varphi}) \land \PR_T(\gn{\neg\varphi})) \mid \varphi\text{ is an }\LA\text{-sentence}\}$.
\end{enumerate}
\end{defn}

We write $T\vdash\ConS_T$ if every sentence in $\ConS_T$ is provable in $T$. 
If $T\vdash\ConS_T$, then $T\vdash\ConL_T$. 
For these derivability conditions and consistency statements, the second incompleteness theorem of the following form holds. 

\begin{fact}[cf.~\cite{Kura25}]\label{G2}\leavevmode
\begin{enumerate}
    \item There is no $\Sigma_1$ provability predicate $\PR_T(x)$ satisfying $\D{2}$, $\D{3}$, and $T\vdash\ConL_T$.
    \item There is no $\Sigma_1$ provability predicate $\PR_T(x)$ satisfying $\mathbf{C}$, $\D{3}$, and $T\vdash\ConS_T$.
\end{enumerate}
\end{fact}

On the other hand, as mentioned in the introduction, Mostowski's provability predicate \cite{Mos65} satisfies $\mathbf{C}$, $\D{3}$, and $T\vdash\ConL_T$.

\subsection{Modal logics}

The language of modal propositional logic consists of countably many propositional variables $p_0,p_1,\ldots$, the logical constant $\bot$, the logical connectives $\neg$, $\land$, $\lor$, $\to$, and the modal operator $\Box$. 
The symbols $\top$ and $\leftrightarrow$ are introduced as abbreviations in the usual way. 
Let $\MF$ be the set of all modal formulas. 
For a modal formula $A$, let $\Sub(A)$ be the set of all subformulas of $A$.

A \textit{modal logic} is a set $L\subseteq\MF$ containing all classical tautologies and closed under Modus Ponens and uniform substitution. 
The pure logic of necessitation $\N$ is the smallest modal logic closed under \textsc{Nec}: $\dfrac{A}{\Box A}$.

Let $\mathsf{C}$ and $\mathsf{4}$ denote the axiom schemata $\Box A\land\Box B\to\Box(A\land B)$ and $\Box A\to\Box\Box A$, respectively. 
These axiom schemata correspond to the derivability conditions $\mathbf{C}$ and $\D{3}$, respectively. 
Let $\mathsf{P}$ be the axiom $\neg\Box\bot$, and let $\mathsf{D}$ be the axiom scheme $\neg(\Box A\land\Box\neg A)$. 
These modal principles correspond to the consistency statements $\ConL_T$ and $\ConS_T$, respectively. 
It is easily shown that $\mathsf{D}$ is stronger than $\mathsf{P}$ over the logic $\N$.

For $\Lambda\subseteq\{\mathsf{N},\mathsf{C},\mathsf{P},\mathsf{D},\mathsf{4}\}$, let $L_\Lambda$ be the smallest modal logic satisfying the following requirements: it is closed under \textsc{Nec} if $\mathsf{N}\in\Lambda$, and it contains each of the axiom schemata among $\mathsf{C}$, $\mathsf{P}$, $\mathsf{D}$, and $\mathsf{4}$ if their names belong to $\Lambda$.
In particular, $L_{\{\N\}} = \N$. 
Although we defined a wider class of non-normal modal logics, we mainly study the following logics:
\begin{align*}
\CN   &= L_{\{\mathsf{N},\mathsf{C}\}},&
\CNP  &= L_{\{\mathsf{N},\mathsf{C},\mathsf{P}\}},&
\CND  &= L_{\{\mathsf{N},\mathsf{C},\mathsf{D}\}},\\
\CNF  &= L_{\{\mathsf{N},\mathsf{C},\mathsf{4}\}},&
\CNPF &= L_{\{\mathsf{N},\mathsf{C},\mathsf{P},\mathsf{4}\}},&
\CNDF &= L_{\{\mathsf{N},\mathsf{C},\mathsf{D},\mathsf{4}\}}.
\end{align*}
Notice that every logic $L_{\Lambda}$ is consistent because it is a sublogic of the normal modal logic $\mathsf{KD4}$.

\subsection{Arithmetical interpretations and provability logics}

Let $\PR_T(x)$ be a provability predicate of $T$.  
A mapping $f$ from modal formulas to $\LA$-sentences is called an arithmetical interpretation based on $\PR_T(x)$ if it satisfies the following conditions:
\begin{enumerate}
\item $f(\bot)$ is $0=1$, 
\item $f(A\circ B)$ is $f(A)\circ f(B)$ for $\circ\in\{\land,\lor,\to\}$, 
\item $f(\neg A)$ is $\neg f(A)$, 
\item $f(\Box A)$ is $\PR_T(\gn{f(A)})$.
\end{enumerate}

\begin{defn}
The provability logic $\PL(\PR_T)$ of $\PR_T(x)$ is the set of all modal formulas $A$ such that $T\vdash f(A)$ for all arithmetical interpretations $f$ based on $\PR_T(x)$.
\end{defn}

It is obvious that $\N\subseteq\PL(\PR_T)$ for any provability predicate $\PR_T(x)$.

\begin{prop}\label{prop:sound-arith}
Let $\PR_T(x)$ be a provability predicate of $T$.
\begin{enumerate}
\item $\PR_T(x)$ satisfies $\mathbf{C}$ if and only if $\CN = L_{\{\N, \mathsf{C}\}} \subseteq\PL(\PR_T)$.

\item $\PR_T(x)$ satisfies $\D{3}$ if and only if $L_{\{\N, \mathsf{4}\}} \subseteq\PL(\PR_T)$.

\item $T\vdash\ConL_T$ if and only if $L_{\{\N, \mathsf{P}\}} \subseteq\PL(\PR_T)$.

\item $T\vdash\ConS_T$ if and only if $L_{\{\N, \mathsf{D}\}} \subseteq\PL(\PR_T)$.
\end{enumerate}
\end{prop}

We have $\CNPF \subseteq \PL(\PR_T^{\mathrm{M}})$ for Mostowski's provability predicate $\PR_T^{\mathrm{M}}(x)$. 
From Fact \ref{G2}.2, $\CNDF$ is not arithmetically sound, that is, there is no $\Sigma_1$ provability predicate $\PR_T(x)$ such that $\CNDF \subseteq \PL(\PR_T)$. 
The main purpose of Section \ref{Sec4} is to prove the arithmetical completeness theorems for $\CN$, $\CNP$, $\CNF$, $\CNPF$, and $\CND$.
That is, for each logic $L$ among them, there exists a $\Sigma_1$ provability predicate $\PR_T(x)$ of $T$ such that $\PL(\PR_T) = L$. 

For textbooks on provability logic, we refer to Boolos \cite{Boolos} and Smory\'nski \cite{Smor85}. 
Comprehensive surveys on provability logic are given by Japaridze and de Jongh~\cite{JD} and Artemov and Beklemishev~\cite{AB}.

\section{A new semantics}\label{Sec3}

In this section, we introduce a new semantics for non-normal modal logics. 
In ordinary Kripke semantics, the truth of $\Box A$ at a world is defined in terms of an accessibility relation. 
In our semantics, each point $x$ is assigned a set $G(x)$ of formulas and a closure operator $\clo$ is fixed on sets of formulas.  A formula $\Box A$ is true at $x$ exactly when $A\in\clo(G(x))$. 
The distinction between the generator $G(x)$ and its closure $\clo(G(x))$ is important below, since $G(x)$ may be finite even if $\clo(G(x))$ is infinite.

\begin{defn}
A \textit{closure operator} is a mapping $\clo:\mathcal{P}(\MF)\to\mathcal{P}(\MF)$ satisfying the following conditions for all $X,Y\subseteq\MF$:
\begin{enumerate}
\item $X\subseteq\clo(X)$, 
\item If $X\subseteq Y$, then $\clo(X)\subseteq \clo(Y)$, 
\item $\clo(\clo(X))=\clo(X)$.
\end{enumerate}
\end{defn}

\begin{defn}
A \textit{closure frame} is a triple $\mathcal{F} = (W,G,\clo)$ where $W$ is a non-empty set, $G: W \to \mathcal{P}(\MF)$, and $\clo$ is a closure operator. 

A closure frame is said to be \textit{finite} if $W$ is finite and $G(x)$ is finite for every $x\in W$.
\end{defn}

\begin{defn}
A \textit{closure model} is a quadruple $\mathcal{M} = (W,G,\clo,\Vdash)$ where $\mathcal{F} = (W,G,\clo)$ is a closure frame and $\Vdash$ is a satisfaction relation between $W$ and $\MF$ satisfying the following clauses:
\begin{enumerate}
\item $x\nVdash\bot$.
\item $x\Vdash \neg A \iff x \nVdash A$.
\item $x\Vdash A \land B \iff$ ($x\Vdash A$ and $x\Vdash B$).
\item $x\Vdash A \lor B \iff$ ($x\Vdash A$ or $x\Vdash B$).
\item $x\Vdash A\to B \iff$ ($x\nVdash A$ or $x\Vdash B$).
\item $x\Vdash\Box A \iff A\in \clo(G(x))$.
\end{enumerate}
A modal formula $A$ is said to be \textit{valid} in a closure model $\mathcal{M}$ (in symbols, $\mathcal{M} \models A$) if $x\Vdash A$ for all $x\in W$. 
Let $\Th(\mathcal{M}) : = \{A \in \MF \mid \mathcal{M} \models A\}$. 
\end{defn}

In the proofs of modal completeness below, countermodels are constructed in different ways depending on $\Lambda\subseteq\{\mathsf{N},\mathsf{C},\mathsf{P},\mathsf{D},\mathsf{4}\}$. 
Some of these countermodels are finite. 
Even if a closure frame is finite, the set $\clo(G(x))$ may be infinite. 
For the proofs of arithmetical completeness in Section~\ref{Sec4}, we need an effective presentation of the relation $B\in\clo(G(x))$ so that arithmetical formulas representing the model can be obtained.

We therefore use the following notion. 
A closure model $\mathcal M=(W,G,\clo,\Vdash)$ is said to be \textit{primitive recursively presented} if the relations $x\in W$, $x\Vdash p$ for propositional variables $p$, and $B\in\clo(G(x))$ are primitive recursive. 
Then the full satisfaction relation $x\Vdash B$ is also primitive recursive, by recursion on the construction of $B$.

\begin{defn}
Let $\mathcal{M}=(W,G,\clo,\Vdash)$ be a closure model. 
We say that $\mathcal{M}$ satisfies:
\begin{enumerate}
\item $(\mathsf{N}) : \iff \forall x \in W \, \forall A \in \MF (\mathcal{M} \models A \Rightarrow A \in \clo(G(x)))$, 

\item $(\mathsf{C}) : \iff \forall x\in W\, \forall A,B \in \MF(A, B \in \clo(G(x)) \Rightarrow A\land B\in \clo(G(x)))$, 

\item $(\mathsf{4}) : \iff \forall x\in W\, \forall A \in \MF (A\in \clo(G(x)) \Rightarrow \Box A\in \clo(G(x)))$, 

\item $(\mathsf{P}) : \iff \forall x \in W (\bot\notin \clo(G(x)))$, 

\item $(\mathsf{D}) : \iff \forall x\in W\, \forall A\in\MF (A\in \clo(G(x)) \Rightarrow \neg A \notin \clo(G(x)))$, 
\end{enumerate}
\end{defn}

\begin{prop}\label{prop:closure-conditions}
Let $\mathcal{M}=(W,G,\clo,\Vdash)$ be a closure model. 
Then $\mathcal{M}$ satisfies $(\mathsf{N})$ if and only if $\Th(\mathcal{M})$ is closed under \textsc{Nec}. 
Moreover, for each $\mathsf X\in\{\mathsf{C},\mathsf{4},\mathsf{D}\}$, the model $\mathcal{M}$ satisfies $(\mathsf X)$ if and only if $\mathsf X\subseteq\Th(\mathcal{M})$, and $\mathcal{M}$ satisfies $(\mathsf{P})$ if and only if $\mathsf{P}\in\Th(\mathcal{M})$.
\end{prop}

In general, the set $\Th(\mathcal{M})$ of formulas is not closed under uniform substitution. 
However, our logics $L_\Lambda$ are defined by axiom schemata, and hence uniform substitution is available as an admissible rule even if it is not equipped as an explicit rule. 
This gives the following soundness result.

\begin{prop}
Let $\Lambda\subseteq\{\mathsf{N},\mathsf{C},\mathsf{P},\mathsf{D},\mathsf{4}\}$. 
If a closure model $\mathcal{M}$ satisfies $(\mathsf{X})$ for all $\mathsf{X} \in \Lambda$, then $L_\Lambda \subseteq \Th(\mathcal{M})$.
\end{prop}

In the following subsections, we prove completeness results with respect to our
new semantics. 
We first study logics in which $\mathsf{N}$ and $\mathsf{D}$ are not both contained. 
In this case, we develop methods for constructing primitive recursively presented finite countermodels.
We then study logics containing both $\mathsf{N}$ and $\mathsf{D}$. 
We prove the completeness theorems for these logics by using a canonical model construction with a different closure operator. 
The models obtained in this way are not primitive recursively presented. 
Therefore, for the arithmetical completeness theorem for $\CND$, we separately prove that $\CND$ has a primitive recursively presented countermodel.

\subsection{Finite model property for logics in which $\mathsf{N}$ and $\mathsf{D}$ are not both contained}

In this subsection, let $\Lambda\subseteq\{\mathsf{N},\mathsf{C},\mathsf{P}, \mathsf{D}, \mathsf{4}\}$ and assume $\{\mathsf{N}, \mathsf{D}\} \nsubseteq \Lambda$. 
For each modal formula $B$, let ${\sim}B$ be $C$ if $B$ is of the form $\neg C$ and ${\sim}B \equiv \neg B$ otherwise. 

\begin{defn}
    Let $A$ be a modal formula. 
\begin{enumerate}
    \item $\Sub^\Box(A) : = \Sub(A)\cup\{\Box B\mid B\in\Sub(A)\} \cup \{\bot, \Box \bot\}$. 

    \item $\Sub^\ast(A) : = \Sub^{\Box}(A)\cup\{{\sim}B\mid B\in\Sub^{\Box}(A)\}$.

    \item For $X\subseteq\MF$, $\clo_{A}(X) : = X \cup \{B\in\MF \mid \Box B\notin\Sub^\Box(A)\}$. 
\end{enumerate}
\end{defn}

It is easily shown that $\clo_A$ is a closure operator. 
We are ready to prove our first completeness result for logics without $\mathsf{D}$.

\begin{thm}\label{compl_withoutD}
Let $\Lambda\subseteq\{\mathsf{N},\mathsf{C},\mathsf{P},\mathsf{4}\}$. 
If $L_\Lambda\nvdash A$, then there exists a finite closure model $\mathcal{M} = (W, G, \clo_A, \Vdash)$ satisfying the conditions $(\mathsf{X})$ for all $\mathsf{X} \in \Lambda$ in which $A$ is not valid.
\end{thm}

\begin{proof}
Suppose that $L_\Lambda\nvdash A$. 
We say that $X \subseteq \Sub^\ast(A)$ is \textit{$L_\Lambda$-consistent} if $L_{\Lambda} \nvdash \bigwedge X \to \bot$. 
$X$ is said to be \textit{$\Sub^\ast(A)$-maximally $L_\Lambda$-consistent} if $X$ is $L_\Lambda$-consistent and for every $B\in\Sub^\ast(A)$, either $B\in X$ or ${\sim} B\in X$. 
It is shown that every $L_{\Lambda}$-consistent set $X \subseteq \Sub^\ast(A)$ can be extended to a $\Sub^\ast(A)$-maximally $L_{\Lambda}$-consistent set. 

We define a quadruple $\mathcal{M}_A = (W, G, \clo_A, \Vdash)$ as follows: 
\begin{itemize}
    \item $W : = \{X \subseteq \Sub^\ast(A) \mid X$ is $\Sub^\ast(A)$-maximally $L_\Lambda$-consistent$\}$. 

    \item For $X\in W$, let $G(X):=\{B\in\Sub^\ast(A) \mid \Box B\in X\}$. 
    
    \item For $X\in W$ and each propositional variable $p$, $X \Vdash p : \iff p\in X$.
\end{itemize}

Since $L_\Lambda$ is consistent, the empty set $\emptyset$ is $L_{\Lambda}$-consistent.
Then, there exists a $\Sub^\ast(A)$-maximally $L_{\Lambda}$-consistent set. 
Thus, $W$ is non-empty. 

\begin{lem}[Truth Lemma]
For every $X\in W$ and every $B\in\Sub^\ast(A)$, we have that $X\Vdash B$ if and only if $B\in X$.
\end{lem}

\begin{proof}
We prove the lemma by induction on the construction of $B$. 
We describe only the case $B=\Box C$. 
By the definition, $X \Vdash \Box C$ is equivalent to $C \in \clo_A(G(X)) = G(X) \cup \{D \in \MF \mid \Box D \notin \Sub^{\Box}(A)\}$. 
Since $\Box C \in \Sub^{\Box}(A)$, we get that $C \in \clo_A(G(X))$ is equivalent to $C \in G(X)$, and this is equivalent to $\Box C \in X$.  
\end{proof}

\begin{lem}
The model $\mathcal{M}_A$ satisfies the conditions $(\mathsf{X})$ for all $\mathsf{X} \in \Lambda$.
\end{lem}

\begin{proof}
$(\N)$: Suppose $\mathsf{N}\in\Lambda$ and $B$ is valid in $\mathcal{M}_A$. 
We prove $B\in \clo_A(G(X))$ for all $X\in W$.
If $\Box B\notin\Sub^\Box(A)$, then $B\in \clo_A(G(X))$ by the definition of $\clo_{A}$. 
If $\Box B\in\Sub^\Box(A)$, then by the Truth Lemma, $\mathcal{M}_A \models B$ implies that $B\in X$ for all $X\in W$. 
This means that $\{{\sim}B\}$ is not $L_{\Lambda}$-consistent, that is, $B \in L_{\Lambda}$. 
By \textsc{Nec}, $\Box B \in L_\Lambda$. 
Thus $\Box B \in X$ for all $X\in W$. 
Then, $B\in G(X) \subseteq \clo_A(G(X))$ for all $X\in W$.

\medskip

$(\mathsf{P})$: Assume $\mathsf{P}\in\Lambda$. 
Suppose, towards a contradiction, that $\bot\in\clo_A(G(X))$. 
Since $\Box\bot\in\Sub^\Box(A)$ by the definition, $\bot\in G(X)$, and so $\Box\bot\in X$. 
This contradicts the axiom $\neg\Box\bot\in L_\Lambda$ and the $L_\Lambda$-consistency of $X$. 
Therefore $\bot\notin\clo_A(G(X))$.

\medskip

$(\mathsf{C})$: Suppose $\mathsf{C}\in\Lambda$ and $B, C\in\clo_A(G(X))$. 
We show that $B\land C\in\clo_A(G(X))$.
If $\Box(B\land C)\notin\Sub^\Box(A)$, then $B\land C\in\clo_A(G(X))$ by the definition of $\clo_A$.
Suppose $\Box(B\land C)\in\Sub^\Box(A)$. 
Then, $B,C\in\Sub(A)$, and hence $\Box B,\Box C\in\Sub^\Box(A)$. 
Since $B,C\in\clo_A(G(X))$, we have $B,C\in G(X)$. 
Thus $\Box B,\Box C\in X$. 
Since $L_\Lambda$ contains the axiom $\Box B\land\Box C\to\Box(B\land C)$, we obtain $\Box(B\land C)\in X$. 
Hence $B\land C\in G(X)\subseteq\clo_A(G(X))$.

\medskip

$(\mathsf{4})$: Suppose $\mathsf{4}\in\Lambda$. 
Let $B\in\clo_A(G(X))$. 
We show that $\Box B\in\clo_A(G(X))$.
If $\Box\Box B\notin\Sub^\Box(A)$, then $\Box B\in\clo_A(G(X))$ by the definition of $\clo_A$.
Suppose $\Box\Box B\in\Sub^\Box(A)$. 
Then $\Box B\in\Sub(A)$, and hence $\Box B\in\Sub^\Box(A)$. 
Thus $B\in G(X)$, so $\Box B\in X$. 
Since $L_\Lambda$ contains the axiom $\Box B\to\Box\Box B$, we obtain $\Box\Box B\in X$. 
Hence $\Box B\in G(X)\subseteq\clo_A(G(X))$.
\end{proof}

Since $L_{\Lambda} \nvdash A$, we have that $\{{\sim}A\}$ is $L_{\Lambda}$-consistent. 
Then, we find a $\Sub^\ast(A)$-maximally $L_{\Lambda}$-consistent set $X_A \in W$ such that $A \notin X_A$. 
By the Truth Lemma, $X_A \nVdash A$. 
Thus $\mathcal{M}_A$ is a required finite countermodel.
\end{proof}

When $\mathsf{N}\notin\Lambda$, it suffices in the proof above to take $W=\{X_A\}$. 
This is because the full set $W$ is used only to verify $(\mathsf{N})$.

We next study the case where $\mathsf{D}\in\Lambda$ and $\mathsf{N}\notin\Lambda$. 
In this case the closure operator $\clo_A$ is clearly no longer available because it violates the condition $(\mathsf{D})$. 
Instead, since $(\mathsf{N})$ is not required, we use a closure operator generated by the clauses corresponding to $\mathsf{C}$ and $\mathsf{4}$.

Let $\Lambda\subseteq\{\mathsf{C},\mathsf{P},\mathsf{D},\mathsf{4}\}$ with $\mathsf{D}\in\Lambda$. 
We define our closure operator $\clo_\Lambda$ as follows. 
For $X\subseteq\MF$, we define a sequence $\{X_n^\Lambda\}_{n\in\omega}$ of sets of modal formulas as follows. 
Let $X_0^\Lambda:=X$. 
Let $X_{n+1}^\Lambda$ be the union of $X_n^\Lambda$ and the following sets:
\begin{itemize}
\item $\{B\land C\mid B,C\in X_n^\Lambda\}$ if $\mathsf{C}\in\Lambda$;
\item $\{\Box B\mid B\in X_n^\Lambda\}$ if $\mathsf{4}\in\Lambda$.
\end{itemize}
Finally, we define $\clo_\Lambda(X):=\bigcup_{n\in\omega}X_n^\Lambda$.
It is easily verified that $\clo_\Lambda$ is a closure operator.

\begin{thm}\label{compl_D_withoutN}
Let $\Lambda\subseteq\{\mathsf{C},\mathsf{P},\mathsf{D},\mathsf{4}\}$ with $\mathsf{D}\in\Lambda$. 
If $L_\Lambda\nvdash A$, then there exists a finite closure model $\mathcal{M}=(W,G,\clo_\Lambda,\Vdash)$ satisfying the conditions $(\mathsf{X})$ for all $\mathsf{X}\in\Lambda$ in which $A$ is not valid.
\end{thm}

\begin{proof}
Suppose $L_\Lambda\nvdash A$. 
Then, we find a $\Sub^\ast(A)$-maximally $L_\Lambda$-consistent set $X_A$ such that $A\notin X_A$. 
We define a finite closure model $\mathcal{M}_A = (W,G,\clo_\Lambda,\Vdash)$ as follows:
\begin{itemize}
\item $W:=\{X_A\}$.

\item Let $G(X_A):=\{B\in\Sub^\ast(A)\mid \Box B\in X_A\}$.

\item For each propositional variable $p$, we define $X_A \Vdash p:\iff p\in X_A$.
\end{itemize}

\begin{lem}\label{DwithoutN_Box_Lem}
For every $B\in\MF$, if $B\in\clo_\Lambda(G(X_A))$, then $\Box B\in X_A$.
\end{lem}

\begin{proof}
We prove the following statement by induction on $n$: for any modal formula $B$, if $B\in G(X_A)_n^\Lambda$, then $\Box B\in X_A$.

If $n=0$, then $B\in G(X_A)$, and hence $\Box B\in X_A$ by the definition of $G(X_A)$.

Suppose that the statement holds for $n$ and that $B\in G(X_A)_{n+1}^\Lambda$. 
If $B\in G(X_A)_n^\Lambda$, then we are done by the induction hypothesis.

If $B$ is introduced by the $\land$-clause, then $B$ is of the form $C\land D$ for some $C,D\in G(X_A)_n^\Lambda$. 
By the induction hypothesis, $\Box C,\Box D\in X_A$. 
Since $\mathsf{C}\in\Lambda$, the axiom $\Box C\land\Box D\to\Box(C\land D)$ belongs to $L_\Lambda$. 
Hence $\Box(C\land D)\in X_A$.

If $B$ is introduced by the $\Box$-clause, then $B$ is of the form $\Box C$ for some $C\in G(X_A)_n^\Lambda$. 
By the induction hypothesis, $\Box C\in X_A$. 
Since $\mathsf{4}\in\Lambda$, the axiom $\Box C\to\Box\Box C$ belongs to $L_\Lambda$. 
Hence $\Box\Box C\in X_A$.
\end{proof}

\begin{lem}[Truth Lemma]\label{DwithoutN_Truth_Lem}
For every $B\in\Sub^\ast(A)$, we have $X_A\Vdash B$ if and only if $B\in X_A$.
\end{lem}

\begin{proof}
We prove the lemma by induction on the construction of $B$. 
The cases of Boolean connectives are standard. 
We only describe the case $B=\Box C$.

By the definition of satisfaction, $X_A\Vdash\Box C$ is equivalent to $C\in\clo_\Lambda(G(X_A))$. 
If $C\in\clo_\Lambda(G(X_A))$, then $\Box C\in X_A$ by Lemma~\ref{DwithoutN_Box_Lem}. 
Conversely, if $\Box C\in X_A$, then $C\in G(X_A)\subseteq\clo_\Lambda(G(X_A))$.
\end{proof}

\begin{lem}\label{DwithoutN_conditions_Lem}
The model $\mathcal{M}_A$ satisfies the conditions $(\mathsf{X})$ for all $\mathsf{X}\in\Lambda$.
\end{lem}

\begin{proof}
$(\mathsf{P})$: Assume $\mathsf{P}\in\Lambda$. 
Suppose, towards a contradiction, that $\bot\in\clo_\Lambda(G(X_A))$. 
Then $\Box\bot\in X_A$ by Lemma~\ref{DwithoutN_Box_Lem}. 
This contradicts the axiom $\neg\Box\bot\in L_\Lambda$ and the $L_\Lambda$-consistency of $X_A$.
Therefore $\bot\notin\clo_\Lambda(G(X_A))$. 

\medskip

$(\mathsf{C})$: Suppose $\mathsf{C}\in\Lambda$ and $B,C\in\clo_\Lambda(G(X_A))$. 
Then $B\land C\in\clo_\Lambda(G(X_A))$ by the $\land$-clause in the definition of $\clo_\Lambda$.

\medskip

$(\mathsf{4})$: Suppose $\mathsf{4}\in\Lambda$ and $B\in\clo_\Lambda(G(X_A))$. 
Then $\Box B\in\clo_\Lambda(G(X_A))$ by the $\Box$-clause in the definition of $\clo_\Lambda$.

\medskip

$(\mathsf{D})$: Suppose, towards a contradiction, that $B,\neg B\in\clo_\Lambda(G(X_A))$. 
Then $\Box B,\Box\neg B\in X_A$ by Lemma~\ref{DwithoutN_Box_Lem}. 
This contradicts the axiom $\neg(\Box B\land\Box\neg B)\in L_\Lambda$ and the $L_\Lambda$-consistency of $X_A$.
Therefore $\clo_\Lambda(G(X_A))$ does not contain both $B$ and $\neg B$. 
\end{proof}

By the Truth Lemma, $X_A\nVdash A$. 
Thus $\mathcal{M}_A$ is a required finite countermodel.
\end{proof}

Since the relations $B \in \clo_A(X)$ in the proof of Theorem \ref{compl_withoutD} and $B \in \clo_{\Lambda}(X_A)$ in the proof of Theorem \ref{compl_D_withoutN} are primitive recursive, our finite countermodels $\mathcal{M}_A$ are actually primitive recursively presented. 
Hence, we conclude the following theorem.

\begin{thm}\label{fmp}
Let $\Lambda\subseteq\{\mathsf{N},\mathsf{C},\mathsf{P}, \mathsf{D}, \mathsf{4}\}$ and assume $\{\mathsf{N}, \mathsf{D}\} \nsubseteq \Lambda$. 
For any modal formula $A$, the following are equivalent: 
\begin{enumerate}
    \item $L_{\Lambda} \vdash A$. 
    \item $A$ is valid in all closure models satisfying the conditions $(\mathsf{X})$ for all $\mathsf{X} \in \Lambda$. 
    \item $A$ is valid in all primitive recursively presented finite closure models satisfying the conditions $(\mathsf{X})$ for all $\mathsf{X} \in \Lambda$. 
\end{enumerate}
\end{thm}

The preceding proofs also yield primitive recursive decision procedures. 
To determine whether $L_\Lambda\vdash A$ or not, it is sufficient to search for a countermodel of the particular form constructed above.
In this search, the closure operator is not part of the data to be searched.
If $\Lambda \subseteq \{\N, \mathsf{C}, \mathsf{P}, \mathsf{4}\}$, the construction uses the fixed closure operator $\clo_A$ in the proof of Theorem \ref{compl_withoutD}. 
If $\Lambda \subseteq \{\mathsf{C}, \mathsf{P}, \mathsf{D}, \mathsf{4}\}$ with $\mathsf{D}\in\Lambda$, the construction uses the fixed closure operator $\clo_\Lambda$ in the proof of Theorem \ref{compl_D_withoutN}.
Thus only the finite data, such as the set of worlds, the generators $G(x)$, and the valuation of propositional variables occurring in $A$, need be searched.

In the case of Theorem~\ref{compl_withoutD}, the finite set $\Sub^\ast(A)$ is primitive recursively computed from $A$. Then, in order to look for a countermodel to $A$, it is sufficient to search through finite sets of worlds $W\subseteq\mathcal{P}(\Sub^\ast(A))$, finite generators $G(x)\subseteq\Sub^\ast(A)$ for $x\in W$, and valuations of propositional variables occurring in $A$. For each $\mathsf{X}\in\Lambda$, by the definition of $\clo_A$, the verification of the condition $(\mathsf{X})$ reduces to checking the behavior of the finite generators $G(x)$ on formulas in $\Sub^\ast(A)$.

In the case of Theorem~\ref{compl_D_withoutN}, in order to look for a countermodel to $A$, it is sufficient to search through a single set $X_A\subseteq\Sub^\ast(A)$, a finite generator $G(X_A)\subseteq\Sub^\ast(A)$, and a valuation of propositional variables occurring in $A$. 
The conditions $(\mathsf{C})$ and $(\mathsf{4})$ are built into the definition of $\clo_\Lambda$. 
For the conditions $(\mathsf{P})$ and $(\mathsf{D})$, the verification is also primitive recursive. 
Indeed, $\clo_\Lambda$ never generates $\bot$ or a formula whose outermost connective is $\neg$. 
Hence, for $(\mathsf{P})$, it is sufficient to check $\bot\notin G(X_A)$.
For $(\mathsf{D})$, it is sufficient to check that $B\notin\clo_\Lambda(G(X_A))$ for every $\neg B\in G(X_A)$.
Since $G(X_A)$ is finite and the relation $B\in\clo_\Lambda(G(X_A))$ is primitive recursive, this gives a primitive recursive check.

Hence we obtain the following theorem.

\begin{thm}\label{decidability}
Let $\Lambda\subseteq\{\mathsf{N},\mathsf{C},\mathsf{P}, \mathsf{D}, \mathsf{4}\}$ and assume $\{\mathsf{N}, \mathsf{D}\} \nsubseteq \Lambda$. 
Then, the set $L_\Lambda$ is primitive recursive. 
\end{thm}

\subsection{Completeness for logics with $\mathsf{N}$ and $\mathsf{D}$}\label{withND}

We next study logics containing both $\mathsf{N}$ and $\mathsf{D}$. 
In the finite countermodel constructions above, the Truth Lemma is proved only for formulas in the finite set $\Sub^\ast(A)$. 
This is not sufficient for verifying $(\mathsf{N})$, since $(\mathsf{N})$ concerns all formulas valid in the model. 
Indeed, in order to satisfy $(\mathsf{N})$, the sets $\clo(G(X))$ must contain sufficiently many formulas. 
On the other hand, $(\mathsf{D})$ requires these sets $\clo(G(X))$ to avoid containing both $B$ and $\neg B$. 
To overcome this situation, we use a canonical model construction in which the Truth Lemma holds for all modal formulas.

Let $\{\mathsf{N},\mathsf{D}\}\subseteq
\Lambda\subseteq\{\mathsf{N},\mathsf{C},\mathsf{D},\mathsf{4}\}$. 
We define our closure operator $\clo_{\Lambda}$ as follows. 
For $X\subseteq\MF$, we define a sequence $\{X_n^\Lambda\}_{n\in\omega}$ of sets of modal formulas as follows. 
Let $X_0^\Lambda:=X\cup L_\Lambda$. 
Let $X_{n+1}^\Lambda$ be the union of $X_n^\Lambda$ and the following sets:
\begin{itemize}
\item $\{B\land C\mid B,C\in X_n^\Lambda\}$ if $\mathsf{C}\in\Lambda$;
\item $\{\Box B\mid B\in X_n^\Lambda\}$ if $\mathsf{4}\in\Lambda$.
\end{itemize}
Finally, we define $\clo_\Lambda(X):=\bigcup_{n\in\omega}X_n^\Lambda$.
It is easily verified that $\clo_{\Lambda}$ is a closure operator. 
The only difference from the definition of $\clo_{\Lambda}$ in the previous subsection is that the logic $L_{\Lambda}$ is contained in the initial set $X_0^\Lambda$.

We define a closure model $\mathcal{M}_\Lambda=(W,G,\clo_\Lambda,\Vdash)$ as follows:
\begin{itemize}
\item $W$ is the set of all maximally $L_\Lambda$-consistent sets of modal formulas.

\item For $X\in W$, let $G(X):=\{B\in\MF\mid \Box B\in X\}$.

\item For $X\in W$ and each propositional variable $p$, we define $X\Vdash p:\iff p\in X$.
\end{itemize}

\begin{lem}\label{ND_Box_Lem}
For every $X\in W$ and every $B\in\MF$, if $B\in\clo_\Lambda(G(X))$, then $\Box B\in X$.
\end{lem}

\begin{proof}
The proof is the same as that of Lemma~\ref{DwithoutN_Box_Lem}, except for the case where $B$ is in $G(X)_0^{\Lambda}$ because $B\in L_\Lambda$. 
In this case, since $\mathsf{N}\in\Lambda$, we have $\Box B \in L_\Lambda$ by \textsc{Nec}. 
Hence $\Box B\in X$.
\end{proof}

\begin{lem}[Truth Lemma]\label{ND_Truth_Lem}
For every $X\in W$ and every modal formula $B$, we have $X\Vdash B$ if and only if $B\in X$.
\end{lem}

\begin{proof}
We prove the lemma by induction on the construction of $B$. 
The cases for Boolean connectives are standard. 
The case $B=\Box C$ is proved exactly as in Lemma~\ref{DwithoutN_Truth_Lem},
using Lemma~\ref{ND_Box_Lem} in place of Lemma~\ref{DwithoutN_Box_Lem}.
\end{proof}

\begin{lem}\label{ND_conditions}
The model $\mathcal{M}_\Lambda$ satisfies the conditions $(\mathsf{X})$ for all $\mathsf{X}\in\Lambda$.
\end{lem}

\begin{proof}
The proofs of $(\mathsf{C})$, $(\mathsf{4})$, and $(\mathsf{D})$ are the same as in the proof of Lemma \ref{DwithoutN_conditions_Lem} using Lemma~\ref{ND_Box_Lem} in place of Lemma~\ref{DwithoutN_Box_Lem}. 
We only prove $(\mathsf{N})$.

$(\mathsf{N})$: Let $B$ be a modal formula valid in $\mathcal{M}_\Lambda$. 
By Lemma~\ref{ND_Truth_Lem}, $B\in X$ for all $X\in W$. 
Hence $B\in L_\Lambda\subseteq\clo_\Lambda(G(X))$ for every $X\in W$. 
\end{proof}

\begin{lem}\label{ND_canonical_completeness}
For every modal formula $A$, $A \in L_{\Lambda}$ if and only if $\mathcal{M}_{\Lambda} \models A$.
\end{lem}

\begin{proof}
$(\Rightarrow)$: If $A \in L_\Lambda$, then $A\in X$ for every maximally $L_\Lambda$-consistent set $X\in W$. 
By the Truth Lemma, $X\Vdash A$ for every $X\in W$. 
Hence $\mathcal{M}_\Lambda\models A$.

\medskip

$(\Leftarrow)$: Suppose that $A \notin L_\Lambda$. 
Then $\{\neg A\}$ is $L_\Lambda$-consistent, and hence there exists a maximally $L_\Lambda$-consistent set $X_A\in W$ such that $A \notin X_A$. 
By the Truth Lemma, $X_A\nVdash A$. 
Thus $\mathcal{M}_\Lambda\not\models A$.
\end{proof}

We conclude the following completeness theorem.

\begin{thm}\label{thm:complete_withND}
Let $\{\mathsf{N},\mathsf{D}\}\subseteq \Lambda\subseteq\{\mathsf{N},\mathsf{C},\mathsf{D},\mathsf{4}\}$. 
For any modal formula $A$, the following are equivalent:
\begin{enumerate}
\item $A \in L_\Lambda$.
\item $A$ is valid in all closure models satisfying the conditions $(\mathsf{X})$ for all $\mathsf{X}\in\Lambda$.
\end{enumerate}
\end{thm}

\begin{prob}\label{prob:withND-fmp}
Let $\{\mathsf{N},\mathsf{D}\}\subseteq \Lambda\subseteq\{\mathsf{N},\mathsf{C},\mathsf{D},\mathsf{4}\}$. 
Does $L_\Lambda$ have the finite model property with respect to closure models satisfying the conditions $(\mathsf{X})$ for all $\mathsf{X}\in\Lambda$?
\end{prob}

\subsection{Primitive recursively presented countermodels for $\CND$}\label{CND}

The main purpose of this subsection is to obtain a primitive recursively presented closure countermodel for $\CND$. 
This is needed in the proof of the arithmetical completeness theorem for $\CND$.
For this purpose, we first prove that $\CND$ is primitive recursive. 

For $S\subseteq\MF$, let $\clo_{\land}(S)$ denote the $\land$-closure of $S$. 

\begin{lem}\label{reduction}
Let $B_0, \ldots, B_{k-1}, C_0, \ldots, C_{l-1}$ be modal formulas and let $D$ be a propositional formula. 
The following are equivalent:
\begin{enumerate}
\item $\CND\vdash \Box B_0 \lor \cdots \lor \Box B_{k-1} \lor \neg \Box C_0 \lor \cdots \lor \neg \Box C_{l-1} \lor D$.
\item One of the following three conditions holds:
\begin{description}
    \item [(i)] For some $i < k$, $B_i \in \clo_{\land}(\{C_0, \ldots, C_{l-1}\} \cup \CND)$, 
    \item [(ii)] $\CND \vdash \neg (C_0 \land \cdots \land C_{l-1})$,
    \item [(iii)] $D$ is a propositional tautology. 
\end{description}
\end{enumerate}
\end{lem}

\begin{proof}
Let $A$ be the formula $\Box B_0 \lor \cdots \lor \Box B_{k-1} \lor \neg \Box C_0 \lor \cdots \lor \neg \Box C_{l-1} \lor D$. 

\medskip

$(1 \Rightarrow 2)$: We prove the contrapositive. 
Assume that none of (i), (ii), and (iii) hold. 
Since $D$ is not a tautology, we find a truth assignment $v$ such that $v(D) = 0$. 
We extend the domain of $v$ to modal formulas as follows, by regarding boxed formulas as atoms: For each modal formula $E$,  
\begin{itemize}
    \item $v(\Box E)=1 : \iff E\in\clo_{\land}(\{C_0, \ldots, C_{l-1}\}\cup\CND)$.
\end{itemize}
The truth value of every modal formula is then defined in the usual way. 

We show that every theorem of $\CND$ is true under $v$. 
This is proved by induction on the length of a proof in $\CND$.
\begin{itemize}
    \item If $E$ is a tautology, then clearly $v(E)=1$.

    \item Suppose that $E$ is an instance $\Box E_0\land\Box E_1\to\Box(E_0\land E_1)$ of the axiom $\mathsf{C}$. 
If $v(\Box E_0)=v(\Box E_1)=1$, then $E_0,E_1\in\clo_{\land}(\{C_0, \ldots, C_{l-1}\}\cup\CND)$. 
Then, $E_0\land E_1\in\clo_{\land}(\{C_0, \ldots, C_{l-1}\}\cup\CND)$, and hence $v(\Box(E_0\land E_1))=1$. 

    \item Suppose that $E$ is an instance $\neg(\Box E_0\land\Box\neg E_0)$ of the axiom $\mathsf{D}$.
If $v(\Box E_0\land\Box\neg E_0)=1$, then $E_0,\neg E_0\in\clo_{\land}(\{C_0, \ldots, C_{l-1}\}\cup\CND)$. 
Then, it is easy to show that $\CND\vdash C_0 \land \cdots \land C_{l-1} \to E_0$ and $\CND\vdash C_0 \land \cdots \land C_{l-1}\to\neg E_0$. 
Hence $\CND\vdash\neg (C_0 \land \cdots \land C_{l-1})$. 
This contradicts the failure of (ii). 
Therefore $v(\neg(\Box E_0\land\Box\neg E_0))=1$.

    \item It is obvious that Modus Ponens is sound under $v$.

    \item We show that \textsc{Nec} is sound under $v$. 
Suppose that $\Box E$ is obtained from $E$ by \textsc{Nec}. 
Since $\CND \vdash E$, we have $E\in \clo_{\land}(\{C_0, \ldots, C_{l-1}\}\cup\CND)$. 
Hence $v(\Box E)=1$.
\end{itemize}
Therefore every theorem of $\CND$ is true under $v$.

Let $i < k$. 
By the failure of (i), we have $B_i \notin \clo_{\land}(\{C_0, \ldots, C_{l-1}\}\cup\CND)$. 
So $v(\Box B_i) = 0$. 
Let $j < l$. 
Since $C_j \in \clo_{\land}(\{C_0, \ldots, C_{l-1}\}\cup\CND)$, we get $v(\Box C_j) = 1$ and so $v(\neg \Box C_j) = 0$. 
Since $v(D) = 0$, we obtain $v(A) = 0$.
Since every theorem of $\CND$ is true under $v$ but $v(A) = 0$, we conclude that $\CND \nvdash A$.

\medskip

$(2 \Rightarrow 1)$: 
Suppose that (i) holds, that is, $B_i \in \clo_{\land}(\{C_0, \ldots, C_{l-1}\} \cup \CND)$ for some $i < k$. 
We have $\CND \vdash \Box C_0 \land \cdots \land \Box C_{l-1} \to \Box B_i$ by using \textsc{Nec} for elements of $\CND$ and repeatedly applying the axiom $\mathsf{C}$.
Then $\CND \vdash \Box B_i \lor \neg \Box C_0 \lor \cdots \lor \neg \Box C_{l-1}$, and hence $\CND \vdash A$. 

\medskip

Suppose that (ii) holds, that is, $\CND \vdash \neg (C_0 \land \cdots \land C_{l-1})$. 
By \textsc{Nec}, $\CND \vdash \Box \neg (C_0 \land \cdots \land C_{l-1})$. 
On the other hand, by the repeated applications of $\mathsf{C}$, we have $\CND \vdash \Box C_0 \land \cdots \land \Box C_{l-1} \to \Box (C_0 \land \cdots \land C_{l-1})$. 
By the axiom $\mathsf{D}$, we obtain $\CND \vdash \neg \Box C_0 \lor \cdots \lor \neg \Box C_{l-1}$. 
Thus, $\CND \vdash A$. 

\medskip

Suppose that (iii) holds, that is, $D$ is a propositional tautology. 
Then $\CND \vdash D$, and hence $\CND \vdash A$ clearly holds. 
\end{proof}

For each modal formula $A$, the \textit{modal degree} $\dg(A)$ of $A$ is the maximum number of nested occurrences of $\Box$ in $A$.
For each $n\in\omega$, let 
\[
    \CND_{\leq n}:=\{C\in\MF\mid \CND\vdash C\ \text{and}\ \dg(C)\leq n\}. 
\]
\begin{lem}\label{restriction}
Let $n \in \omega$. 
Let $B$ be a modal formula with $\dg(B) \leq n$ and $S$ be a set of modal formulas whose modal degrees are at most $n$. 
The following are equivalent: 
\begin{enumerate}
    \item $B \in \clo_{\land}(S \cup \CND)$. 
    \item $B \in \clo_{\land}(S \cup \CND_{\leq n})$. 
\end{enumerate}
\end{lem}

\begin{proof}
Since the operator $\clo_{\land}$ does not increase the modal degrees of elements, if $B \in \clo_{\land}(S \cup \CND)$, then $B$ is generated only by elements of $S$ and elements of $\CND$ with the modal degrees at most $n$. 
So, we have $B \in\clo_{\land}(S \cup \CND_{\leq n})$.
Thus the direction $(1 \Rightarrow 2)$ holds. 
The direction $(2 \Rightarrow 1)$ is obvious. 
\end{proof}

\begin{lem}\label{CND_pr}
The set $\CND$ is primitive recursive.
\end{lem}

\begin{proof}
We prove by induction on $n$ that the set $\CND_{\leq n}$ is primitive recursive.
For $n=0$, the set $\CND_{\leq 0}$ is clearly primitive recursive.

Assume that the set $\CND_{\leq n}$ is primitive recursive. 
Let $A$ be a modal formula with $\dg(A)\leq n+1$. 
By the usual primitive recursive procedure, we compute from $A$ a modal formula $A_0\land\cdots\land A_m$ in modal conjunctive normal form such that $\CND\vdash A\leftrightarrow A_0\land\cdots\land A_m$ and $\dg(A_i)\leq n+1$ for all $i\leq m$.
Then, $\CND \vdash A$ is equivalent to $\CND \vdash A_i$ for all $i \leq m$. 
So, it suffices to show that the $\CND$-provability of each $A_i$ is primitive recursively determined. 

Each $A_i$ is of the form $\Box B_0 \lor \cdots \lor \Box B_{k-1} \lor \neg \Box C_0 \lor \cdots \lor \neg \Box C_{l-1} \lor D$, where $\dg(B_i) \leq n$ for all $i < k$, $\dg(C_j) \leq n$ for all $j < l$, and $\dg(D) = 0$. 
From Lemma \ref{reduction} and Lemma \ref{restriction}, we have that $\CND \vdash A_i$ if and only if one of the following three conditions holds: 
\begin{description}
    \item [(i)] For some $i < k$, $B_i \in \clo_{\land}(\{C_0, \ldots, C_{l-1}\} \cup \CND_{\leq n})$, 
    \item [(ii)] $\neg (C_0 \land \cdots \land C_{l-1}) \in \CND_{\leq n}$,
    \item [(iii)] $D$ is a propositional tautology. 
\end{description}
By the induction hypothesis, each of these conditions is primitive recursively decidable because $\clo_{\land}$ is generated only by conjunction.
Therefore $\CND_{\leq n+1}$ is primitive recursive. 
\end{proof}

We construct a primitive recursively presented closure model $\mathcal{M}_{\CND}'$ for $\CND$. 
Let $\clo_{\CND}$ be the closure operator defined in Subsection~\ref{withND} for the case $\Lambda=\{\mathsf{N},\mathsf{C},\mathsf{D}\}$. 
Since $\CND$ is primitive recursive by Lemma~\ref{CND_pr}, for each modal formula $C$ with $\CND\nvdash\neg C$, we can primitive recursively construct a maximally $\CND$-consistent set $X_C$ extending $\{C\}$ in the usual way.

We define a closure model $\mathcal{M}_{\CND}':=(W,G,\clo,\Vdash)$ as follows: 
\begin{itemize}
\item $W:=\{C\in\MF\mid \CND\nvdash\neg C\}$.

\item For $C\in W$, let $G(C):=\{B\in\MF\mid \Box B\in X_C\}$.

\item Let $\clo$ be the closure operator $\clo_{\{\N, \mathsf{C}, \mathsf{D}\}}$ introduced in Subsection \ref{withND}. 

\item For $C\in W$ and each propositional variable $p$, we define $C\Vdash p:\iff p\in X_C$.
\end{itemize}

Then it can be verified that the model $\mathcal{M}_{\CND}'$ is primitive recursively presented.
The following lemmas are proved in the same way as in Subsection~\ref{withND}.

\begin{lem}[Truth Lemma]\label{lem:CND-canonical-truth}
For every $C\in W$ and every $B \in \MF$, we have $C\Vdash B$ if and only if $B\in X_C$.
\end{lem}

\begin{lem}\label{lem:CND-canonical-conditions}
The closure model $\mathcal{M}_{\CND}'$ satisfies $(\mathsf{N})$, $(\mathsf{C})$, and $(\mathsf{D})$.
\end{lem}

\begin{lem}\label{CND_canonical}
For every modal formula $A$, $\CND \vdash A$ if and only if $\mathcal{M}_{\CND}'\models A$.
\end{lem}

We conclude the following completeness theorem.

\begin{thm}\label{thm:complete_CND}
For any modal formula $A$, the following are equivalent:
\begin{enumerate}
\item $\CND \vdash A$.
\item $A$ is valid in all closure models satisfying the conditions $(\mathsf{N})$, $(\mathsf{C})$, and $(\mathsf{D})$.
\item $A$ is valid in all primitive recursively presented closure models satisfying the conditions $(\mathsf{N})$, $(\mathsf{C})$, and $(\mathsf{D})$.
\end{enumerate}
\end{thm}

\section{Arithmetical completeness}\label{Sec4}

In this section, we prove the arithmetical completeness theorem for $\CN$, $\CNP$, $\CNF$, $\CNPF$, and $\CND$ using our results established so far.
Here, we say that a logic $L$ is \textit{arithmetically complete} if $L = \PL(\PR_T)$ for some $\Sigma_1$ provability predicate $\PR_T(x)$ of $T$. 
Before proving the theorem, we introduce several notions, which are used throughout the rest of this paper. 
The method developed in this section is based on techniques that the authors have developed in earlier work \cite{Kogure1,Kogure2,KK,KK2,Kura24}.

Let $\{ \xi_t \}_{t \in \omega}$ be the repetition-free primitive recursive enumeration of all $\LA$-formulas in ascending order of G\"odel numbers. 
We note that if $\xi_u$ is a proper subformula of $\xi_v$, then $u<v$.
We call an $\LA$-formula \textit{propositionally atomic} if it is either atomic or of the form $Q x \psi$, where $Q \in \{\forall, \exists \}$ and $\psi$ is an arbitrary $\LA$-formula.
For each propositionally atomic formula $\varphi$, we prepare a distinct propositional variable $p_{\varphi}$.
We define a primitive recursive injection $I$ from the set of all $\LA$-formulas into a set of propositional formulas as follows:
\begin{itemize}
\item 
$I(\varphi) = p_{\varphi}$ for each propositionally atomic formula $\varphi$,
\item 
$I(\varphi \circ \psi) = I(\varphi) \circ I(\psi)$ for $\circ \in \{ \wedge, \vee, \to \}$, 
\item 
$I(\neg \varphi) = \neg I(\varphi)$. 
\end{itemize}
Let $X$ be a finite set of $\LA$-formulas.
An $\LA$-formula $\varphi$ is called a \textit{tautological consequence} (\textit{t.c.}) of $X$
if $\bigwedge_{\psi \in X}I(\psi) \to I(\varphi)$ is a tautology.  
Let $X \tc \varphi$ denote that $\varphi$ is a t.c.~of $X$.
For each $n \in \omega$, let $F_n$ be the set of all $\LA$-formulas whose G\"odel numbers are less than or equal to $n$. 
We may assume that $F_0 = \emptyset$. 
Let
\[
    P_{T,n} : = \{ \varphi \mid \mathbb{N} \models \exists y \leq \num{n} \ \Proof_T(\gn{\varphi}, y)\},
\]
where $\mathbb{N}$ is the standard model of arithmetic and $\Proof_T(x, y)$ is a standard primitive recursive proof predicate of $T$ naturally expressing that ``$y$ is the G\"{o}del number of a proof of $x$ from $T$''.
Let $\Prov_T(x)$ be the standard provability predicate defined by $\exists y \Proof_T(x,y)$.
If $P_{T,n} \tc \varphi$, then $\varphi$ is clearly provable in $T$.
Notice that $P_{T, n} \subseteq F_n$. 
These facts about the above notions can be formalized and verified in $\PA$.

Next, we prepare a primitive recursive function $h$, which was introduced in \cite{Kura20b}.
The function $h$ is defined step by step using the Recursion Theorem.
It starts with the value $0$. 
If $h(m)=0$ and the associated set $J_m$ is non-empty, then $h(m+1)$ becomes non-zero. 
Once the value becomes non-zero, it remains unchanged thereafter. 

\begin{itemize}
	\item $h(0) = 0$. 
	\item $h(m+1) = \begin{cases} \text{min} \ J_m & \text{if}\ h(m) = 0 \ \text{and} \ J_m \neq \emptyset, \\
				
			h(m) & \text{otherwise},
		\end{cases}$
\end{itemize}
where $J_m = \{ j \in \omega \setminus \{0\} \mid P_{T,m} \tc \neg \lambda(\num{j}) \}$ and $\lambda(x)$ is the $\Sigma_1$ formula $\exists y(h(y) = x)$. 
The function $h$ satisfies the property that for each $m$ and $i$, $h(m+1)=i$ implies $i \leq m+1$, which guarantees that the function $h$ is primitive recursive (See \cite{Kura20b}, p. 603).
The following proposition holds for $h$. 
\begin{prop}[Cf.~{\cite[Lemma 3.2]{Kura20b}}]\label{Prop:h}
\leavevmode
\begin{enumerate}
	\item $\PA \vdash \forall x \forall y(0 < x < y \land \lambda(x) \to \neg \lambda(y))$. 
	\item $\PA \vdash \neg \Con_T \leftrightarrow \exists x(\lambda(x) \land x \neq 0)$. 
	\item For each $i \in \omega \setminus \{0\}$, $T \nvdash \neg \lambda(\num{i})$. 
	\item For each $l \in \omega$, $\PA \vdash \forall x \forall y(h(x) = 0 \land h(x+1) = y \land y \neq 0 \to x \geq \num{l})$. 
\end{enumerate}
\end{prop}

\subsection{The logics $\CN$ and $\CNP$}

In this subsection, we prove the arithmetical completeness of $\CN$ and
$\CNP$.

\begin{thm}\label{thmCN}
Let $L \in \{\CN, \CNP\}$. 
There exists a $\Sigma_1$ provability predicate $\PR_T(x)$ of $T$ satisfying the following properties: 
\begin{enumerate}
    \item \textup{(Arithmetical soundness)} 
    For any $A \in \MF$ and any arithmetical interpretation $f$ based on $\PR_T(x)$, if $L \vdash A$, then $\PA \vdash f(A)$; 
    \item \textup{(Uniform arithmetical completeness)} 
    There exists an arithmetical interpretation $f$ based on $\PR_T(x)$ such that for any $A \in \MF$, $L \vdash A$ if and only if $T \vdash f(A)$.
\end{enumerate}
\end{thm}
\begin{proof}
Let $L \in \{\CN, \CNP\}$. 
By Theorem \ref{decidability}, we have a primitive 
recursive enumeration $\langle A_k  \rangle_{k \in \omega}$ of all $L$-unprovable formulas.
For each $k \in \omega$, by Theorem \ref{compl_withoutD}, we can construct a primitive recursively presented finite closure model $\bigl( W_k, G_k, \clo_{A_k}, \Vdash_k \bigr)$ which falsifies $A_k$ and 
satisfies the conditions corresponding to $L$.
We may assume that the sets $\langle  W_k \rangle_{k \in \omega}$ are pairwise disjoint subsets of $\omega$ and $\bigcup_{k \in \omega} W_k = \omega \setminus \{0\}$. 
We may also assume that for each $i \in \omega \setminus \{0\}$, we can primitive recursively find a unique $k \in \omega$ satisfying $i \in W_k$.

We define the primitive recursive function $g_0$ outputting all $T$-provable formulas step by step.
The definition of $g_0$ consists of Procedures 1 and 2, and starts with Procedure 1.
In Procedure 1, we define the values $g_0(0), g_0(1), \ldots$ 
 by referring to the values $h(0), h(1), \ldots$ and $T$-proofs based on $\Proof_T(x,y)$. 
At the first time $h(m+1) \neq 0$, the definition of $g_0$ switches to Procedure 2 at Stage $m$.  
 
We define the primitive recursive function $g_0$ by using the recursion theorem.  
In the definition of $g_0$, we can use the $\Sigma_1$ formula $\PR_{g_0}(x) \equiv$ $\exists y (g_0(y)=x)$ and the arithmetical interpretation $f_0$ based on $\PR_{g_0}(x)$ such that 
\[
    f_0(p) \equiv \exists x \exists y(\lambda(x) \wedge x \neq 0 \land  x \in W_y \wedge x \Vdash_y p).
\]
Here, for each $A \in \MF$, $f_0(A)$ is primitive recursively computed from $A$. Also, $A$ is primitive recursively computed from $f_0(A)$ because $f_0$ is an injection.

\medskip

\textbf{Procedure 1.}

Stage $m$:
\begin{itemize}
\item If $h(m+1) =0$,
\begin{equation*}
  g_0(m)  = \begin{cases}
       \varphi & \text{if}\ m \text{ is a}\ T \text{-proof of}\ \varphi, \\
               0 & \text{otherwise}.
             \end{cases}
  \end{equation*}

Then, go to Stage $m+1$.

\item If $h(m+1) \neq 0$, go to Procedure 2.
\end{itemize}

\textbf{Procedure 2.}

Suppose $h(m)=0$ and $h(m+1)=i \neq 0$. 
Let $k$ be the number such that $i \in W_k$.
Define 
\begin{equation*}
	g_0(m+t)= \begin{cases} \xi_t & \text{if}\  \xi_t \equiv f_0(B)\  \&\ i \Vdash_k \Box B \ \text{for some}\  B \in \MF, \\
    & \text{or } \xi_t \equiv \varphi \land \psi, \text{ for some } \varphi, \psi \in \{g_0(0), \ldots,g_0(m+t-1) \}, \\
			                    
		0 & \text{otherwise}.
			\end{cases}
	\end{equation*}

We have completed the definition of $g_0$.

\medskip

The following claim ensures that $\PR_{g_0}(x)$ becomes a provability predicate of $T$.
\begin{cl}\label{cl:4-1}
$\PA + \Con_T \vdash \forall x(\PR_{g_0}(x)\leftrightarrow\Prov_T(x))$.
\end{cl}

\begin{proof}
By the construction of Procedure 1 and Proposition \ref{Prop:h}.2, we obtain $\PA + \Con_T \vdash g_0(y)=x \leftrightarrow \Proof_T(x,y)$. 
Thus, $\PA + \Con_T \vdash \PR_{g_0}(x) \leftrightarrow \Prov_T(x)$ holds.
\end{proof}

The following claim guarantees that $\PR_{g_0}(x)$ satisfies the condition $\mathbf{C}$.

\begin{cl}\label{cl:4-2}
For any $\LA$-formulas $\varphi$ and $\psi$, $\PA \vdash \PR_{g_0}(\gn{\varphi}) \land \PR_{g_0}(\gn{\psi}) \to \PR_{g_0}(\gn{\varphi \land \psi})$.
\end{cl}
\begin{proof}
Let $\varphi$ and $\psi$ be any $\LA$-formulas.
Since $\Prov_T(x)$ satisfies $\mathbf{C}$,
 we obtain $\PA + \Con_T \vdash \PR_{g_0}(\gn{\varphi}) \land \PR_{g_0}(\gn{\psi}) \to \PR_{g_0}(\gn{\varphi \land \psi})$ by Claim \ref{cl:4-1}.

We prove $\PA + \neg \Con_T \vdash \PR_{g_0}(\gn{\varphi}) \land \PR_{g_0}(\gn{\psi}) \to \PR_{g_0}(\gn{\varphi \land \psi})$.
We argue in $\PA + \neg \Con_T$: By Proposition \ref{Prop:h}.2, there exist numbers $m$ and $i$ such that $i \neq 0$, $h(m)=0$, and $h(m+1) = i$.
Suppose $\PR_{g_0}(\gn{\varphi})$ and $\PR_{g_0}(\gn{\psi})$ hold, that is, $\varphi$ and $\psi$ are output by $g_0$.
If $\varphi$ is output in Procedure 1, 
then
$\varphi = g_0(l)$ for some $l<m$.
Let $u$ and $s$ be such that  $\xi_u \equiv \varphi$ and $\xi_s \equiv \varphi \land \psi$.
If $\varphi$ is output in Procedure 2, then
$\varphi = g_0(m+s)$ holds.
In either case, it follows that $\varphi \in \{ g_0(0), \ldots, g_0(m+u) \}$.
Since the G\"{o}del number of $\varphi$ is less than that of $\varphi \land \psi$, by the choice of the enumeration $\{\xi_t\}_{t \in \omega}$,
we obtain $\varphi\in \{g_0(0),\ldots,g_0(m+s-1)\}$. 
The same argument applied to $\psi$ yields $\psi\in \{g_0(0),\ldots,g_0(m+s-1)\}$.
Therefore, by the definition of Procedure 2, we have
$g_0(m+s)=\xi_s\equiv \varphi\land\psi$, that is, $\PR_{g_0}(\gn{\varphi \land \psi})$ holds.

By the law of excluded middle, it follows that $\PA \vdash \PR_{g_0}(\gn{\varphi}) \land \PR_{g_0}(\gn{\psi}) \to \PR_{g_0}(\gn{\varphi \land \psi})$.

\end{proof}

\begin{cl}\label{cl:4-3}
If $L=\CNP$, then $\PA \vdash\neg\PR_{g_0}(\gn{0=1})$.
\end{cl}

\begin{proof}
By Claim \ref{cl:4-1}, we obtain $\PA + \Con_T \vdash \neg \PR_{g_0}(\gn{0=1})$.
We prove $\PA + \neg \Con_T \vdash \neg \PR_{g_0}(\gn{0=1})$.
We argue in $\PA + \neg \Con_T$: 
Let $m$, $i$, and $k$ be such that $h(m)=0$, $h(m+1)=i \in W_k$.
Suppose, towards a contradiction, that $\PR_{g_0}(\gn{0=1})$ holds.
We distinguish the following two cases.

If $0=1$ is output by $g_0$ in Procedure 1, then $0=1 \in P_{T,m-1}$.
Thus, $P_{T,m-1}$ is inconsistent, and $P_{T,m-1} \tc \neg \lambda(\num{i})$ holds.
It follows that $J_{m-1} \neq \emptyset$, and hence $h(m) \neq 0$, which is a contradiction.

Suppose $0=1$ is output by $g_0$ in Procedure 2.
Since $0=1$ is not of the form $\varphi \land \psi$, by the definition of $g_0$, 
the formula  $0=1$ is $f_0(A)$ and $i \Vdash_k \Box A$ for some $A \in \MF$.
Since $f_0$ is injective and $f_0(\bot) \equiv 0=1$, the modal formula $A$ is $\bot$.
Hence, we obtain $i \Vdash_k \Box \bot$, which is a contradiction because
 $(W_k, G_k,\clo_{A_k},\Vdash_k)$ satisfies $(\mathsf{P})$ and $i \nVdash_k \Box \bot$. 
Therefore, $0=1$ is not output by $g_0$ in Procedure 2. 
We conclude that $\neg \PR_{g_0}(\gn{0=1})$ holds.
\end{proof}

\begin{cl}\label{cl:4-4}
Let $i \in W_k$ and $B \in \MF$.
\begin{enumerate}
	\item
	If $i \Vdash_k B$, then $\PA \vdash \lambda(\num{i}) \to f_{0}(B)$.
	\item 
	If $i \nVdash_k B$, then $\PA \vdash \lambda(\num{i}) \to \neg f_{0}(B)$.
\end{enumerate}
\end{cl}
\begin{proof}
We prove Clauses 1 and 2 simultaneously by induction on the number of symbols occurring in $B$.
We prove the base case of the induction. 
If $B$ is $\bot$, then Clauses 1 and 2 trivially hold. 
Suppose that $B \equiv p$.
	\begin{enumerate}
	\item  
	If $i \Vdash_k p$, then $\PA \vdash \lambda (\num{i}) \to \bigl(\lambda(\num{i}) \land \num{i} \neq 0 \land  \num{i} \in W_k \wedge \num{i} \Vdash_k p \bigr)$ holds.
	Hence, we get $\PA \vdash \lambda(\num{i}) \to f_0(p)$.
	\item 
	Suppose $i \nVdash_k p$. By Proposition \ref{Prop:h}.1, $\PA \vdash \lambda(\num{i}) \to \forall x (\lambda(x) \wedge x \neq 0 \to x= \num{i})$.
	Thus, we obtain $\PA \vdash \lambda(\num{i}) \to \forall x  (\lambda(x) \wedge x \neq 0 \to x \nVdash_k p)$ and hence $\PA \vdash \lambda(\num{i}) \to \neg f_0(p)$.
	\end{enumerate}
	Next, we prove the induction cases. Since the cases of $\neg$, $\wedge$, $\vee$, and $\to$ can be  easily proved, we consider the case $B \equiv \Box C$.

\medskip

1. Suppose $i \Vdash_k \Box C$. We argue in $\PA + \lambda(\num{i})$: Let $m$ and $k$ be such that $h(m)=0$ and $h(m+1)=i \in W_k$.
Let $s$ be the number satisfying $\xi_s \equiv f_0(C)$.
Since $i \Vdash_k \Box C$, by the definition of $g_0$ we obtain $g_0(m+s) = \xi_s \equiv f_0(C)$. 
Thus, $\PR_{g_0}(\gn{f_0(C)})$, that is, $f_0(\Box C)$ holds.

2. Suppose $i \nVdash_k \Box C$. 
Then, $C \notin \clo_{A_k}(G_k(i))$. 
We distinguish the following two cases.

\medskip

Case 1: $C$ is not of the form $D \land E$.

Since $(W_k, G_k, \clo_{A_k},\Vdash_k)$ satisfies $(\mathsf{N})$, $C$ is not valid in $(W_k, G_k, \clo_{A_k},\Vdash_k)$, that is, $j \nVdash_k C$ for some $j \in W_k$.
By the induction hypothesis, we obtain $\PA \vdash \lambda(\num{j}) \to \neg f_0(C)$.
Let $p$ be a $T$-proof of $\lambda(\num{j}) \to \neg f_0(C)$.

We argue in $\PA + \lambda(\num{i})$:
Let $m$ be such that $h(m)=0$ and $h(m+1)=i \in W_k$.
Suppose, towards a contradiction, that $f_0(C)$ is output by $g_0$.
We distinguish the following two cases.
\begin{enumerate}
    \item[(i)]  $f_0(C)$ is output in Procedure 1:
 Then, $f_0(C) \in P_{T,m-1}$.   
By Proposition \ref{Prop:h}.4, $p < m$. 
So, $\lambda(\num{j}) \to \neg f_0(C) \in P_{T,m-1}$.
It follows from them that $P_{T,m-1} \tc \neg \lambda(\num{j})$.
Hence we obtain $J_{m-1} \neq \emptyset$, and thus $h(m) \neq 0$. 
This is a contradiction.
    \item[(ii)] $f_0(C)$ is output in Procedure 2:
Since $f_0(C)$ is not of the form $\varphi \land \psi$,
 there exists $D \in \MF$ such that $f_0(C) \equiv f_0(D)$ and $i \Vdash_k \Box D$.
Since $f_0$ is injective, we obtain $C \equiv D$ and $i \Vdash_k \Box C$, a contradiction.
\end{enumerate}
Therefore, $\neg \PR_{g_0}(\gn{f_0(C)})$, that is, $\neg f_0(\Box C)$ holds.

\medskip

Case 2: $C$ is of the form $D \land E$.

Since $i \nVdash_k \Box (D \land E)$ and $(W_k, G_k, \clo_{A_k},\Vdash_k)$ satisfies $(\mathsf{C})$, we have that $i \nVdash_k \Box D$ or $i \nVdash_k \Box E$.
Without loss of generality, we may assume that $i \nVdash_k \Box D$.
The induction hypothesis applied to $\Box D$ yields 
\begin{equation}\label{1}
\PA \vdash \lambda(\num{i}) \to \neg f_0(\Box D).
\end{equation}

We argue in $\PA + \lambda(\num{i})$: 
Let $m$ be such that $h(m)=0$ and $h(m+1)=i \in W_k$.
Suppose, towards a contradiction, that  $f_0(C)$ is output by $g_0$.
If $f_0(C)$ is output by $g_0$ in Procedure 1, then we obtain a contradiction by the same argument as in (i) above.

Suppose $f_0(C)$ is output by $g_0$ in Procedure 2.
We distinguish the following two cases.
\begin{itemize}
    \item $f_0(C) \equiv f_0(F)$ and $i \Vdash_k \Box F$ for some $F \in \MF$.  
    Then we obtain $i \Vdash_k \Box C$ as in (ii), a contradiction.
    \item $f_0(C) \equiv \varphi \land \psi$  and $\varphi, \psi \in \{g_0(0), \ldots, g_0(m+s-1)\}$ for some $\varphi$ and $\psi$, where $\xi_s \equiv f_0(C)$.
    
Then, $f_0(C)\equiv (f_0(D) \land f_0(E)) \equiv \varphi \land \psi$, and we obtain $f_0(D) \equiv \varphi$.
Then $f_0(D) \in \{ g_0(0), \ldots, g_0(m+s-1) \}$ and thus $\PR_{g_0}(\gn{f_0(D)})$ holds, which contradicts (\ref{1}).
\end{itemize}
Therefore, $f_0(C)$ is not output by $g_0$, that is, $\neg f_0(\Box C)$ holds. 
\end{proof}
We complete our proof of Theorem \ref{thmCN}. 
By Claims \ref{cl:4-2} and \ref{cl:4-3}, the arithmetical soundness stated in Theorem \ref{thmCN} is obvious.
We prove the implication $(\Leftarrow)$ of Clause 2 of Theorem \ref{thmCN}. 
Suppose $L \nvdash A$. 
Then, $A \equiv A_k$ for some $k \in \omega$ and $i \nVdash_k A$ for some $i \in W_k$. Hence, we obtain $\PA \vdash \lambda(\num{i}) \to \neg f_0(A)$ by Claim \ref{cl:4-4}. Thus, we obtain $T \nvdash f_0(A)$ by Proposition \ref{Prop:h}.3.
\end{proof}

\begin{cor}[The arithmetical completeness of $\CN$]\label{AC_CN}
\begin{align*}
    \CN & = \bigcap \{ \PL(\PR_T) \mid \PR_T(x) \text{ is a provability predicate satisfying } \mathbf{C}  \},\\
    & = \bigcap \{ \PL(\PR_T) \mid \PR_T(x) \text{ is a } \Sigma_1 \text{ provability predicate satisfying }  \mathbf{C} \}.
\end{align*}
Moreover, there exists a $\Sigma_1$ provability predicate $\PR_T(x)$ of $T$ such that $\CN = \PL(\PR_T)$.
\end{cor}

\begin{cor}[The arithmetical completeness of $\CNP$]\label{AC_CNP}
\begin{align*}
    \CNP & = \bigcap \{ \PL(\PR_T) \mid \PR_T(x) \text{ satisfies } \mathbf{C} \text{ and } T \vdash \neg \PR_T(\gn{0=1}) \},\\
    & = \bigcap \{ \PL(\PR_T) \mid \PR_T(x) \text{ is } \Sigma_1 \text{ and satisfies } \mathbf{C} \text{ and } T \vdash \neg \PR_T(\gn{0=1}) \}.
\end{align*}
Moreover, there exists a $\Sigma_1$ provability predicate $\PR_T(x)$ of $T$ such that $\CNP = \PL(\PR_T)$.
\end{cor}

\subsection{The logics $\CNF$ and $\CNPF$}

We next prove the arithmetical completeness of $\CNF$ and $\CNPF$. 

\begin{thm}\label{thmCNF}
Let $L \in \{\CNF, \CNPF \}$.
There exists a $\Sigma_1$ provability predicate $\PR_T(x)$ of $T$ satisfying the following properties: 
\begin{enumerate}
    \item \textup{(Arithmetical soundness)} 
    For any $A \in \MF$ and any arithmetical interpretation $f$ based on $\PR_T(x)$, if $L \vdash A$, then $\PA \vdash f(A)$; 
    \item \textup{(Uniform arithmetical completeness)} 
    There exists an arithmetical interpretation $f$ based on $\PR_T(x)$ such that for any $A \in \MF$, $L \vdash A$ if and only if $T \vdash f(A)$.
\end{enumerate}
\end{thm}
\begin{proof}
Let $L \in \{\CNF, \CNPF \}$.
By Theorem \ref{decidability}, we have a primitive recursive enumeration $\langle A_k  \rangle_{k \in \omega}$ of all $L$-unprovable formulas.
For each $k \in \omega$, by Theorem \ref{compl_withoutD}, we can construct a primitive recursively presented finite closure model $(W_k, G_k, \clo_{A_k}, \Vdash_k)$ which falsifies $A_k$ and satisfies the conditions corresponding to $L$.
We may assume that the sets $\langle  W_k \rangle _{k \in \omega}$ are pairwise disjoint subsets of $\omega$ and $\bigcup_{k \in \omega} W_k = \omega \setminus \{0\}$. 

In this proof, we use the primitive recursive function $h$ and the $\Sigma_1$ formula $\lambda(x)$ defined in the previous subsection. 
By using the recursion theorem, we define a primitive recursive function $g_1$ outputting all theorems of $T$.
In the definition of $g_1$, we can use the $\Sigma_1$ formula $\PR_{g_1}(x) \equiv$ $\exists y (g_1(y)=x)$ and the arithmetical interpretation $f_1$ based on $\PR_{g_1}(x)$ such that
\[
    f_1(p) \equiv \exists x \exists y (\lambda(x) \wedge x \neq 0 \land  x \in W_y \wedge x \Vdash_y p).
\]

\textbf{Procedure 1.}

Stage $m$:
\begin{itemize}
\item If $h(m+1) =0$,
\begin{equation*}
  g_1(m)  = \begin{cases}
       \varphi & \text{if}\ m\ \text{is a}\ T \text{-proof of}\ \varphi \\
               0 & \text{otherwise}.
             \end{cases}
  \end{equation*}

Then, go to Stage $m+1$.

\item If $h(m+1) \neq 0$, go to Procedure 2.
\end{itemize}

\textbf{Procedure 2.}

Suppose $h(m)=0$ and $h(m+1)=i \neq 0$. 
Let $k$ be the number such that $i \in W_k$.
Define 
\begin{equation*}
	g_1(m+t)= \begin{cases} \xi_t & \text{if}\  \xi_t \equiv f_1(B)\  \&\ i \Vdash_k \Box B \ \text{for some}\  B \in \MF, \\
    & \text{or } \xi_t \equiv \varphi \land \psi \text{ for some } \varphi, \psi \in \{g_1(0), \ldots,g_1(m+t-1) \}, \\
    & \text{or } \xi_t \equiv \PR_{g_1}(\gn{\varphi}) \text{ for some } \varphi \in \{g_1(0), \ldots,g_1(m+t-1) \},  \\
			                    
		0 & \text{otherwise}.
			\end{cases}
	\end{equation*}

We have completed the definition of $g_1$.

The following two claims can be proved in the same way as Claims \ref{cl:4-1} and \ref{cl:4-2}, respectively.

\begin{cl}\label{cl:4-5}
$\PA + \Con_T \vdash \forall x(\PR_{g_1}(x)\leftrightarrow\Prov_T(x))$.
\end{cl}

\begin{cl}\label{cl:4-6}
For any $\LA$-formulas $\varphi$ and $\psi$, $\PA \vdash \PR_{g_1}(\gn{\varphi}) \land \PR_{g_1}(\gn{\psi}) \to \PR_{g_1}(\gn{\varphi \land \psi})$.
\end{cl}

\begin{cl}\label{cl:4-7}
For any $\LA$-formula $\varphi$,
$\PA \vdash \PR_{g_1}(\gn{\varphi}) \to \PR_{g_1}(\gn{\PR_{g_1}(\gn{\varphi})})$.
\end{cl}
\begin{proof}
Since $\Prov_T(x)$ satisfies $\D{3}$, we obtain $\PA + \Con_T \vdash \PR_{g_1}(\gn{\varphi}) \to \PR_{g_1}(\gn{\PR_{g_1}(\gn{\varphi})})$ by Claim \ref{cl:4-5}.

We prove $\PA + \neg \Con_T \vdash \PR_{g_1}(\gn{\varphi}) \to \PR_{g_1}(\gn{\PR_{g_1}(\gn{\varphi})})$.
We argue in $\PA + \neg \Con_T$:
Let $m$ be such that $h(m)=0$ and $h(m+1)=i \in W_k$.
Suppose $\PR_{g_1}(\gn{\varphi})$ holds, that is, $\varphi$ is output by $g_1$.
Let $s$ be such that $\xi_s \equiv \varphi$.

If $\varphi$ is output in Procedure 1, then $g_1(l) =\varphi$ for some $l<m$.
If $\varphi$ is output in Procedure 2, then 
$\varphi = g_1(m+s)$.
In either case, we obtain $\varphi \in \{g_1(0), \ldots, g_1(m+s)  \}$.
Let $u$ be such that $\xi_u \equiv \PR_{g_1}(\gn{\varphi})$.
Since the G\"{o}del number of $\varphi$ is less than that of $\PR_{g_1}(\gn{\varphi})$,
we obtain $s<u$ and $\varphi \in \{g_1(0), \ldots, g_1(m+u) \}$ by the choice of the enumeration $\{ \xi_t\}_{t \in \omega}$.
Hence, we obtain $g_1(m+u) = \PR_{g_1}(\gn{\varphi})$, and $\PR_{g_1}(\gn{\PR_{g_1}(\gn{\varphi})})$ holds.
\end{proof}

\begin{cl}\label{cl:4-8}
If $L=\CNPF$, then $\PA \vdash\neg\PR_{g_1}(\gn{0=1})$.
\end{cl}
\begin{proof}
This claim can be proved in the same way as Claim \ref{cl:4-3}, taking into account that $0=1$ is not of the form $\PR_{g_1}(\gn{\varphi})$.
\end{proof}

\begin{cl}\label{cl:4-9}
Let $i \in W_k$ and $B \in \MF$.
\begin{enumerate}
\item
If $i \Vdash_k B$, then $\PA \vdash \lambda(\num{i}) \to f_{1}(B)$.
\item 
If $i \nVdash_k B$, then $\PA \vdash \lambda(\num{i}) \to \neg f_{1}(B)$.
\end{enumerate}
\end{cl}

\begin{proof}
We prove Clauses 1 and 2 simultaneously by induction on the number of symbols occurring in $B$.
We only give a proof of the case  $B \equiv \Box C$.
We can prove Clause 1 similarly as in the proof of Claim \ref{cl:4-4}. 
We prove Clause 2.
We distinguish the following two cases.

\medskip

 Case 1: $C$ is of the form $\Box D$.

Suppose $i \nVdash_k \Box C$, that is, $i \nVdash_k \Box \Box D$.
Since $(W_k, G_k, \clo_{A_k}, \Vdash_k)$ satisfies $(\mathsf{4})$, we have $i \nVdash_k \Box D$. 
By the induction hypothesis, we obtain
\begin{equation}\label{2}
\PA \vdash \lambda(\num{i}) \to \neg f_1(\Box D).
\end{equation}
Let $p$ be a $T$-proof of $\lambda(\num{i}) \to \neg f_1(\Box D)$.
We argue in $\PA + \lambda(\num{i})$:
Let $m$ be such that $h(m)=0$ and $h(m+1)=i \in W_k$.
If $f_1(\Box D)$ is output in Procedure 1, then $f_1(\Box D) \in P_{T,m-1}$.
By Proposition \ref{Prop:h}.4, we obtain $m>p$ and $\lambda(\num{i}) \to \neg f_1(\Box D) \in P_{T,m-1}$.
Hence, $P_{T,m-1} \tc \neg \lambda(\num{i})$.
Thus, we obtain $h(m)=0$, a contradiction.

Suppose $f_1(\Box D)$ is output in Procedure 2. Then, $g_1(m+t)= \xi_t \equiv f_1(\Box D)$.
If $f_1(\Box D) \equiv f_1(E)$ and $i \Vdash_k \Box E$ for some $E \in \MF$, then it follows that $E \equiv \Box D$.
Hence, we obtain $i \Vdash_k \Box \Box D$, which contradicts $i \nVdash_k \Box \Box D$.
Thus, it follows that $f_1(\Box D) \equiv \PR_{g_1}(\gn{\varphi})$ and $g_1(l)=\varphi$ for some $l < m+t$ and $\varphi$.
Then, $\PR_{g_1}(\gn{\varphi})$, that is, $f_1(\Box D)$ holds, which contradicts (\ref{2}).

\medskip

Case 2: $C$ is not of the form $\Box D$.

In this case, the claim can be proved in the same way as in the proof of Claim \ref{cl:4-4}.2, taking into account that $f_1(C)$ is not of the form $\PR_{g_1}(\gn{\varphi})$.








Therefore, $f_1(C)$ is not output by $g_1$, that is, $\neg f_1(\Box C)$ holds.
\end{proof}

We complete our proof of Theorem \ref{thmCNF}. 
By Claims \ref{cl:4-6} and \ref{cl:4-7}, the arithmetical soundness stated in Theorem \ref{thmCNF} is obvious.
The implication $(\Leftarrow)$ of Clause 2 can be proved in the same way as in the proof of Theorem \ref{thmCNF}. 
\end{proof}

\begin{cor}[The arithmetical completeness of $\CNF$]\label{AC_CN4}
\begin{align*}
    \CNF & = \bigcap \{ \PL(\PR_T) \mid \PR_T(x) \text{ is a provability predicate satisfying } \mathbf{C} \text{ and } \D{3} \},\\
    & = \bigcap \{ \PL(\PR_T) \mid \PR_T(x) \text{ is a } \Sigma_1 \text{ provability predicate satisfying } \mathbf{C} \text{ and }  \D{3} \}.
\end{align*}
Moreover, there exists a $\Sigma_1$ provability predicate $\PR_T(x)$ of $T$ such that $\CNF = \PL(\PR_T)$.
\end{cor}

\begin{cor}[The arithmetical completeness of $\CNPF$]\label{AC_CNPF}
\begin{align*}
    \CNPF & = \bigcap \{ \PL(\PR_T) \mid \PR_T(x) \text{ satisfies } \mathbf{C}, \D{3}, \text{ and } T \vdash \neg \PR_T(\gn{0=1}) \},\\
    & = \bigcap \{ \PL(\PR_T) \mid \PR_T(x) \text{ is } \Sigma_1 \text{ and satisfies } \mathbf{C}, \D{3}, \text{ and } T \vdash \neg \PR_T(\gn{0=1}) \}.
\end{align*}
Moreover, there exists a $\Sigma_1$ provability predicate $\PR_T(x)$ of $T$ such that $\CNPF = \PL(\PR_T)$.
\end{cor}

\subsection{The logic $\CND$}

We finally prove the arithmetical completeness theorem for $\CND$.  In this
subsection, we use a primitive recursively presented countermodel for
$\CND$ given in Section~\ref{Sec3}.

\begin{thm}\label{thmCND} 
There exists a $\Sigma_1$ provability predicate $\PR_T(x)$ of $T$ satisfying the following properties: 
\begin{enumerate}
    \item \textup{(Arithmetical soundness)} 
    For any $A \in \MF$ and any arithmetical interpretation $f$ based on $\PR_T(x)$, if $\CND \vdash A$, then $\PA \vdash f(A)$; 
    \item \textup{(Uniform arithmetical completeness)} 
    There exists an arithmetical interpretation $f$ based on $\PR_T(x)$ such that for any $A \in \MF$, $\CND \vdash A$ if and only if $T \vdash f(A)$.
\end{enumerate}
\end{thm}
\begin{proof}
By Theorem \ref{CND_pr}, we have a primitive recursive enumeration $\langle A_k  \rangle_{k \geq 1}$ of all $\CND$-unprovable formulas.
Let $\mathcal{M}_{\CND}' = \bigl( W, G, \clo, \Vdash \bigr)$ be the primitive recursively presented closure model for $\CND$ given in Subsection \ref{CND}. 
By the construction of the model, we may assume that $W = \omega \setminus \{0\}$ and $k \nVdash A_k$.


In this proof, we use the function $h'$, which is originally introduced in \cite{Kogure1}. 
The function $h'$ is defined by using the recursion theorem as follows: 

\begin{itemize}
	\item $h'(0) = 0$. 
	\item $h'(m+1) = \begin{cases} \text{min} \ J'_m & \text{if}\ h'(m) = 0 \ \text{and} \ J'_m \neq \emptyset, \\ 
						
		h'(m) & \text{otherwise},
		\end{cases}$
\end{itemize}
where 
\begin{align*}
J'_m = \{ j \in &  W  \mid  P_{T,m} \tc \neg \lambda'(\num{j}) \ \text{or}\ \\
&   \exists B[ B \in \Sub(A_j) \ \& \ P_{T,m} \tc \forall x \varphi_B(x) \wedge \bigl (\varphi_B(\num{j}) \to \neg \lambda'(\num{j}) \bigr)]\}.
\end{align*}
Here, $\lambda'(x)$ and $\varphi_B(x)$ are formulas 
\[
\exists y( h'(y)=x) \ \text{and} \ ( x \neq 0 \land x \nVdash B  ) \to \neg \lambda'(x),
\]
respectively.

The function $h'$ satisfies the property that for each $m$ and $i$, $h'(m+1)=i$ implies $i \leq m+1$, which guarantees that the function $h'$ is primitive recursive (See \cite[Claim 5.1]{Kogure1}).

The following proposition holds for $h'$.

\begin{prop}[Cf.~{\cite[Proposition 5.1]{Kogure1}}]\label{prop:h'}
	\leavevmode

\begin{enumerate}
\item 
$\PA \vdash \forall x \forall y \bigl( 0 < x < y \wedge \lambda'(x) \to \neg \lambda'(y) \bigr)$.
\item 
$\PA \vdash \neg \Con_{T} \leftrightarrow \exists x \bigl(\lambda'(x) \wedge x \neq 0 \bigr)$.
\item 
For each $i \in \omega \setminus \{ 0\}$, $T \nvdash \neg \lambda'(\num{i})$.
\item 
For each $l \in \omega$, $\PA \vdash \forall x \forall y \bigl( h'(x) =0 \wedge h'(x+1)=y \wedge y \neq 0 \to x \geq l \bigr)$.
\end{enumerate}
\end{prop}

\medskip








			                    


Next, by the recursion theorem, we define a primitive recursive function
$g_2$ outputting all theorems of $T$. 
In this definition we may use the formula $\PR_{g_2}(x)\equiv \exists y(g_2(y)=x)$ and the arithmetical interpretation $f_2$ based on $\PR_{g_2}(x)$ given by
\[
    f_2(p)\equiv
    \exists x(\lambda'(x)\land x\neq 0\land x\Vdash p).
\]
The definition of $g_2$ is obtained from the definition of $g_0$ in the proof
of Theorem~\ref{thmCN} by replacing $h$ and $f_0$ with $h'$ and $f_2$, respectively.

The following two claims can be proved in the same way as Claims \ref{cl:4-1} and \ref{cl:4-2}, respectively.
\begin{cl}\label{cl:4.10}
$\PA+ \Con_T\vdash \forall x(\PR_{g_2}(x)\leftrightarrow\Prov_T(x))$.
\end{cl}

\begin{cl}\label{cl:4.11}
For any $\LA$-formulas $\varphi$ and $\psi$, $\PA \vdash \PR_{g_2}(\gn{\varphi}) \land \PR_{g_2}(\gn{\psi}) \to \PR_{g_2}(\gn{\varphi \land \psi})$.
\end{cl}

\begin{cl}\label{cl:4.12}
Let $B \in \MF$.
\begin{enumerate}
\item 
$\PA \vdash \exists x ( x \neq 0 \land \lambda'(x)   \wedge x \Vdash B)  \to f_2(B)$.
\item 
$\PA \vdash \exists x ( x \neq 0 \land \lambda'(x) \wedge  x \nVdash B)   \to \neg f_2(B)$,

 that is, $\PA \vdash \exists x \neg \varphi_B(x) \to \neg f_2(B)$.
\end{enumerate}
\end{cl}
\begin{proof}
We prove Clauses 1 and 2 simultaneously by induction on the number of symbols occurring in $B$.
In the case $B \equiv \bot$, Clauses 1 and 2 are easily proved.
Suppose $B \equiv p$. Clause 1 holds trivially. We only give a proof of Clause 2. 
By Proposition \ref{prop:h'}.1, we obtain 
\begin{align*}
\PA \vdash x \neq 0 \land \lambda'(x) \wedge x \nVdash p &\to ( \lambda'(z) \wedge z\neq 0 \to z=x).
\end{align*}
It follows that 
\[
\PA \vdash \exists x ( x \neq 0 \land  \lambda'(x) \wedge x \nVdash p  ) \to \neg f_2(p).
\]
Next, we prove the induction cases. 
We only prove the case $B \equiv \Box C$.
Clause 1 can be proved in the same way as in Claim \ref{cl:4-4}.1.
We prove Clause 2.
We distinguish the following two cases.

\medskip

Case 1. $C$ is of the form $D \land E$:

By the induction hypothesis applied to $\Box D$ and $\Box E$,
 we obtain 
\begin{equation}\label{(3)}
\PA \vdash \exists x \neg \varphi_{D}(x) \to \neg f_2(\Box D), \quad \PA \vdash \exists x \neg \varphi_{E}(x) \to \neg f_2(\Box E).
\end{equation}

If $\CND \vdash C$, 
then $\CND \vdash \Box C$.
Hence, $\PA \vdash \forall x (x \neq 0 \to x \Vdash \Box C)$ holds by formalizing the soundness of $\CND$.
Thus, Clause 2 trivially holds.
Hence, we may assume that $\CND \nvdash C$.
Then, there exists $j \in W$ such that $C \equiv A_j$ and $j \nVdash C$. 
It follows that
$\PA \vdash \num{j} \neq 0 \wedge \num{j} \nVdash C$, which implies that $\PA \vdash \varphi_C(\num{j}) \to \neg \lambda'(\num{j})$. 
Also, by the induction hypothesis, $\PA \vdash \exists x \neg \varphi_C(x) \to \neg f_2(C)$ holds.
Let $p$ and $q$ be $T$-proofs of $\varphi_C(\num{j}) \to \neg \lambda'(\num{j})$ and $\exists x \neg \varphi_C(x) \to \neg f_2(C)$ respectively.

We work in $\PA$: 
Let $i$ be such that $i \neq 0$, $\lambda'(\num{i})$, and $i \nVdash \Box C$.
 It follows that $i \nVdash \Box D$ or $i \nVdash \Box E$ because $(W,G,\clo,\Vdash)$ satisfies $(\mathsf{C})$ and $C \equiv D \land E$.
Without loss of generality, we may assume that $i \nVdash \Box D$.
Let $m$ be such that $h'(m)=0$ and $h'(m+1)=i$. 
Suppose, towards a contradiction, that $f_2(C)$ is output by $g_2$.
We consider the following two cases.
\begin{itemize}
    \item $f_2(C)$ is output in Procedure 1.

Then, we obtain $f_2(C) \in P_{T, m-1}$, and by Proposition \ref{prop:h'}.4, $q \leq m-1$. 
Hence, we obtain $\exists x \neg \varphi_C(x) \to \neg f_2(C)\in P_{T,m-1}$, which implies $P_{T, m-1} \tc \forall x \varphi_C(x)$.
Since $p \leq m-1$, $P_{T,m-1} \tc \varphi_C(\num{j}) \to \neg \lambda'(\num{j})$.
We also have $j \in W$ and $C \in \Sub(A_j)$.
Hence, it follows that $h'(m) \neq 0$, a contradiction.

    \item $f_2(C)$ is output in Procedure 2.

If $f_2(C) \equiv f_2(F)$ and $i \Vdash \Box F$ for some $F \in \MF$, then we obtain $C \equiv F$ because $f_2$ is injective, which contradicts $i \nVdash \Box C$.
If $f_2(C) \equiv \varphi \land \psi$  and $\varphi, \psi \in \{g_2(0), \ldots, g_2(m+s-1)\}$ for some $\varphi$ and $\psi$ where $\xi_s \equiv f_2(C)$,
then $f_2(C)\equiv (f_2(D) \land f_2(E)) \equiv \varphi \land \psi$, and we obtain $f_2(D) \equiv \varphi$.
We obtain $\PR_{g_2}(\gn{f_2(D)})$, that is, $f_2(\Box D)$ holds. 
This contradicts (\ref{(3)}).    
\end{itemize}
Thus, in either case, it follows that $\neg \PR_{g_2}(\gn{f_2(C)})$, that is, $\neg f_2(\Box C)$ holds.

\medskip

Case 2. $C$ is not of the form $D \land E$:

Case 2 can be proved in a similar way as in the proof of Case 1, taking into account that $f_2(C)$ is not of the form $\varphi \land \psi$.
\end{proof}

\begin{cl}\label{cl:4.13}
For any $\LA$-formula $\varphi$ that is not in the range of $f_2$, 
$\PA$ proves: ``if $h'(m)=0$, $h'(m+1)\neq 0$, and $\PR_{g_2}(\gn{\varphi})$ holds, then $P_{T,m-1} \tc \varphi$.''
\end{cl}
\begin{proof}
We prove the claim by induction on the G\"{o}del number of $\varphi$.
Assume that the claim has already been established for all formulas with smaller G\"{o}del numbers than $\varphi$.
We argue in $\PA$: Suppose $h'(m)=0$, $h'(m+1) \neq 0$, and $\PR_{g_2}(\gn{\varphi})$ holds.
If $\varphi$ is output in Procedure 1, then $\varphi \in P_{T,m-1}$.

If $\varphi$ is output in Procedure 2, then there exist $\psi$ and $\rho$ such that $\varphi \equiv \psi \land \rho$, $\PR_{g_2}(\gn{\psi})$, and $\PR_{g_2}(\gn{\rho})$ hold because $\varphi$ is not in the range of $f_2$.
It follows from the induction hypothesis that $P_{T,m-1} \tc \psi$ and $P_{T,m-1} \tc \rho$.
Thus, we obtain $P_{T,m-1} \tc \psi \land \rho$, that is, $P_{T,m-1} \tc \varphi$.
\end{proof}

\begin{cl}\label{cl:4.14}
For any $\LA$-formula $\varphi$,
$\PA \vdash \neg (\PR_{g_2}(\gn{\varphi}) \wedge \PR_{g_2}(\gn{\neg \varphi}))$.
\end{cl}
\begin{proof}
Since $\PA + \Con_T \vdash \Prov_T(\gn{\varphi}) \to \neg \Prov_T(\gn{\neg \varphi})$, 
it follows from Claim \ref{cl:4.10} that $\PA + \Con_T \vdash \PR_{g_2}(\gn{\varphi}) \to \neg \PR_{g_2}(\gn{\neg \varphi})$. 
We prove $\PA + \neg \Con_T \vdash \PR_{g_2}(\gn{\varphi}) \to \neg \PR_{g_2}(\gn{\neg \varphi})$.
We distinguish the following two cases:

\medskip

Case 1: $\varphi \equiv f_2(B)$ for some $B \in \MF$.

By Proposition \ref{prop:h'}.2, it suffices to prove $\PA + \exists x (x \neq 0 \land \lambda'(x)) \vdash f_2(\Box B) \to \neg f_2(\Box \neg B)$. 
By Claim \ref{cl:4.12}.2, we obtain
    \[
    \PA \vdash x \neq 0 \land \lambda'(x)  \land x \nVdash \Box B \to \neg f_2(\Box B).
    \] 
Then, 
$\PA \vdash x \neq 0 \land \lambda'(x)  \land f_2(\Box B) \to x \Vdash \Box B$.
Since $\PA$ proves that $(W,G,\clo,\Vdash)$ satisfies $(\mathsf{D})$,
\[
    \PA \vdash x \neq 0 \land \lambda'(x)  \land f_2(\Box B) \to x \nVdash \Box \neg B. 
\]
By combining this with Claim \ref{cl:4.12}.2, 
\[
    \PA \vdash x \neq 0 \land \lambda'(x)  \land f_2(\Box B) \to \neg f_2(\Box \neg B). 
\]
Hence, we obtain $\PA + \exists x(x \neq 0 \land \lambda'(x)) \vdash f_2(\Box B) \to \neg f_2(\Box \neg B)$. 

\medskip

Case 2: $\varphi \not \equiv f_2(B)$ for all $B \in \MF$. 

We argue in $\PA + \neg \Con_T$:
Let $m$ and $i$ be such that $h'(m)=0$ and $h'(m+1) =i \neq 0$.
Suppose, towards a contradiction, that 
$\PR_{g_2}(\gn{\varphi})$ and $\PR_{g_2}(\gn{\neg \varphi})$ hold.
By Claim \ref{cl:4.13}, we obtain $P_{T,m-1} \tc \varphi$ and $P_{T,m-1} \tc \neg \varphi$.
Hence, it follows that $P_{T,m-1} \tc \neg \lambda'(\num{i})$, and we obtain $h'(m) \neq 0$, a contradiction.
Therefore, $\neg (\PR_{g_2}(\gn{\varphi}) \land \PR_{g_2}(\gn{\neg \varphi}))$ holds.
\end{proof}
We prove Theorem \ref{thmCND}. 
Clause 1 and the implication $(\Rightarrow)$ of Clause 2 follow from Claim \ref{cl:4.11} and Claim \ref{cl:4.14}.  
We prove the implication $(\Leftarrow)$ of Clause 2. 
Suppose $\CND \nvdash A$. 
Then, there exists $k \in \omega \setminus \{0\}$ such that $A \equiv A_k$ and $k \nVdash A$.
 We obtain $\PA \vdash \num{k} \neq 0 \land \num{k} \nVdash A$. 
 It follows that $\PA \vdash \varphi_A (\num{k}) \to \neg \lambda'(\num{k})$.
By using the contrapositive of Claim \ref{cl:4.12}.2, we obtain
$\PA \vdash f_2(A) \to \forall x \varphi_A (x)$, which implies $\PA \vdash f_2(A)  \to \varphi_A (\num{k})$.
Then, we get $\PA \vdash f_2(A) \to \neg \lambda'(\num{k})$.
From Proposition \ref{prop:h'}.3, we conclude $T \nvdash f_2(A)$.
\end{proof}

\begin{cor}[The arithmetical completeness of $\CND$]\label{AC_CND}
\begin{align*}
    \CND & = \bigcap \{ \PL(\PR_T) \mid \PR_T(x) \text{ satisfies } \mathbf{C} \text{ and }  T \vdash \ConS_T \},\\
    & = \bigcap \{ \PL(\PR_T) \mid \PR_T(x) \text{ is } \Sigma_1 \text{ and satisfies } \mathbf{C} \text{ and }  T \vdash \ConS_T \}.
\end{align*}
Moreover, there exists a $\Sigma_1$ provability predicate $\PR_T(x)$ of $T$ such that $\CND = \PL(\PR_T)$.
\end{cor}

\section{Concluding remarks}

In this paper, we investigated modal logics corresponding to provability
predicates satisfying the derivability condition $\mathbf{C}$. 
We introduced a new semantics in which $\Box A$ is interpreted by membership of $A$ in a set of formulas associated with each point. 
Closure under conjunction in these sets corresponds to the modal axiom $\mathsf{C}$. 

We proved completeness results with respect to our new semantics for the logics determined by combinations of the principles $\mathsf{N}$, $\mathsf{C}$, $\mathsf{P}$, $\mathsf{D}$, and $\mathsf{4}$. 
If $\mathsf{N}$ and $\mathsf{D}$ are not both contained, then our construction yields primitive recursively presented finite countermodels.
For logics containing both $\mathsf{N}$ and $\mathsf{D}$, we proved completeness by using a canonical model construction. 
In the case of $\CND$, we further obtained primitive recursively presented countermodels.

The semantics introduced in this paper seems to be a useful tool for analyzing other weak non-normal logics not discussed in this paper. 
It would also be natural to extend the decidability analysis developed in Subsection \ref{CND} for $\CND$ to other logics, such as $\CNDF$.

We also proved the arithmetical completeness theorems for $\CN$, $\CNP$, $\CNF$,
$\CNPF$, and $\CND$. 
The problem of proving the arithmetical completeness of these logics was suggested in our overview presented in~\cite{KK2}. 
The present paper solves it by introducing the new semantics and by developing the construction of primitive recursively presented countermodels. 
In particular, the arithmetical completeness of these logics shows that the corresponding combinations of derivability conditions do not produce additional modal principles through arithmetical principles such as the Fixed Point Theorem.

It remains to clarify a significance of these results from the viewpoint of the second incompleteness theorem. 
By contrast with the arithmetical completeness results above, $\CNDF$ is not arithmetically sound. 
Thus it would be interesting to investigate arithmetically complete extensions of $\CNPF$. 
One example in this direction is given by Mostowski's provability predicate $\PR_T^{\mathrm{M}}(x)$. 
Its provability logic $\PL(\PR_T^{\mathrm{M}})$ is studied in a forthcoming paper by the second author.

\section*{Acknowledgments}

The first author was supported by JST SPRING, Grant Number JPMJSP2148. 
The second author was supported by JSPS KAKENHI Grant Number JP23K03200.

\bibliographystyle{plain}
\bibliography{refs}

\end{document}